\documentclass[11pt]{amsart}

\usepackage[utf8x]{inputenc}
\usepackage[a4paper]{geometry}
\usepackage{graphicx}
\usepackage[all,cmtip]{xy}
\usepackage{amssymb,amsfonts,latexsym}
\usepackage{enumerate}
\usepackage{multirow}
\usepackage{hyperref} \hypersetup{colorlinks=true, linktoc=page, pdfauthor={Giovanni Stagliano'}, pdftitle={On special quadratic birational transformations} }
\usepackage{mathptmx}

\newcommand{\sS}{{\mathbf{S}}}  % hypersurface image
\newcommand{\B}{{\mathfrak{B}}}  % base locus
\newcommand{\I}{{\mathcal{I}}}   % ideal sheaf
\renewcommand{\O}{{\mathcal{O}}} % structural sheaf
\newcommand{\T}{{\mathcal{T}}} % tangent bundle
\newcommand{\E}{{\mathcal{E}}} % vector bundle
\newcommand{\N}{{\mathcal{N}}} % normal bundle
\newcommand{\CC}{{\mathbb{C}}}   % complex numbers
\newcommand{\FF}{\mathbb{F}}    % rational ruled surface
\newcommand{\GG}{{\mathbb{G}}}   % grassmannian
\newcommand{\ZZ}{{\mathbb{Z}}}   % integers
\newcommand{\PP}{{\mathbb{P}}}   % projective space
\newcommand{\Sec}{{\mathrm{Sec}}} % secant variety
\newcommand{\Bl}{{\mathrm{Bl}}}   % blow-up
\newcommand{\reg}{{\mathrm{reg}}}  %regular locus
\newcommand{\sing}{{\mathrm{sing}}} % singular locus
\newcommand{\Pic}{{\mathrm{Pic}}} % picard group 

\makeatletter

\newcommand{\Rmnum}[1]{\expandafter\@slowromancap\romannumeral #1@}
\makeatother

\theoremstyle{plain}
\newtheorem{proposition}{Proposition}[section]
\newtheorem{theorem}[proposition]{Theorem}
\newtheorem{lemma}[proposition]{Lemma}
\newtheorem{corollary}[proposition]{Corollary}
\newtheorem{fact}[proposition]{Fact}

\theoremstyle{definition}

\newtheorem{assumption}[proposition]{Assumption}

\newtheorem{example}[proposition]{Example}

\theoremstyle{remark}
\newtheorem{remark}[proposition]{Remark}
\newtheorem{claim}{Claim}[proposition]
\newtheorem{case}{Case}[proposition]

\numberwithin{equation}{section}

\title[On special quadratic birational transformations]{On special quadratic birational transformations whose base locus has dimension at most three}
\author[G. Staglian\`o]{Giovanni Staglian\`o} 
\address{Giovanni Staglian\`o -
         Ph.D. in Mathematics (\Rmnum{23} cycle) -  University of Study of Catania 
         } 
\email{\href{mailto:gstagliano@dmi.unict.it}{gstagliano@dmi.unict.it}}
\date{\today} 
\begin{document}
\begin{abstract}
We study birational transformations  
$\varphi:\PP^n\dashrightarrow\overline{\varphi(\PP^n)}\subseteq\PP^N$
defined by linear systems of quadrics
 whose base locus is 
smooth and irreducible of dimension $\leq3$ and whose 
  image $\overline{\varphi(\PP^n)}$ is sufficiently regular.
\end{abstract}
\maketitle
\setcounter{secnumdepth}{1}
\setcounter{tocdepth}{1}
\tableofcontents
\section*{Introduction}\label{sec: introduction} 
In this note we continue the study of special quadratic birational transformations 
$\varphi:\PP^n\dashrightarrow\sS:=\overline{\varphi(\PP^n)}\subseteq\PP^{N}$
started in \cite{note}, 
by reinterpreting techniques and well-known results on 
special Cremona transformations 
(see \cite{crauder-katz-1989}, 
     \cite{crauder-katz-1991}, 
     \cite{ein-shepherdbarron} and
     \cite{hulek-katz-schreyer}).
While in \cite{note} we required that $\sS$ was a hypersurface,
here we allow more freedom in the choice of $\sS$, 
but we only treat the case in which 
the dimension of the base 
locus $\B$ is $r=\dim(\B)\leq3$.
In the last section, we shall also obtain partial results
 in the case $r=4$.

Note that for every closed subscheme $X\subset\PP^{n-1}$ 
cut out by the quadrics containing it, we can 
consider $\PP^{n-1}$ as a hyperplane  in $\PP^n$
and hence $X$ as a subscheme of $\PP^n$. So 
the linear system $|\I_{X,\PP^n}(2)|$ of all quadrics in $\PP^n$
containing $X$ defines a quadratic rational map
 $\psi:\PP^n\dashrightarrow\PP^N$ 
($N=h^0(\I_{X,\PP^n}(2))-1=n+h^0(\I_{X,\PP^{n-1}}(2))$),
which is birational onto the image 
and whose inverse is defined by linear forms, 
i.e. $\psi$ is of type $(2,1)$.
Conversely, every birational transformation 
$\psi:\PP^n\dashrightarrow\overline{\psi(\PP^n)}\subseteq\PP^N$
of type $(2,1)$ whose image is nondegenerate, normal and linearly normal 
arise in this way.
From this it follows that there are many (special) quadratic transformations. 
However, 
when the image $\sS$ of the 
 transformation $\varphi$
 is sufficiently regular,
by straightforward generalization of 
\cite[Proposition~2.3]{ein-shepherdbarron}, 
we obtain  
strong numerical and geometric restrictions on the base locus $\B$. 
For example, as soon as $\sS$ is not too much singular,
 the secant variety $\Sec(\B)\subset\PP^n$ has to be a hypersurface and $\B$
has to be a  $QEL$-variety of type $\delta=\delta(\B)=2\dim(\B)+2-n$; in particular 
$n\leq 2\dim(\B)+2$ and $\Sec(\B)$ is a hyperplane if and only if 
 $\varphi$ is of type $(2,1)$.
So the classification of 
 transformations $\varphi$ of type $(2,1)$ whose base locus 
has dimension $\leq 3$ 
essentially follows  
from classification results on
 $QEL$-manifold:
\cite[Propositions~1.3 and 3.4]{russo-qel1}, 
\cite[Theorem~2.2]{ionescu-russo-conicconnected} and
\cite[Theorems~4.10 and 7.1]{ciliberto-mella-russo}.

When $\varphi$ is of type $(2,d)$ with $d\geq2$, then $\Sec(\B)$ 
is a nonlinear hypersurface 
and it is not so easy to exhibit examples. 
The most difficult cases of this kind 
are those for which $n=2r+2$ i.e. $\delta=0$.
In order to classify these transformations, we first determine the Hilbert
polynomial of $\B$ in Lemmas \ref{lemma: r=2 B nondegenerate} and 
\ref{lemma: r=3 B nondegenerate}, 
by using  
the usual Castelnuovo's argument, Castelnuovo's bound and 
some refinement of Castelnuovo's bound, 
see \cite{ciliberto-hilbertfunctions} and \cite{mella-russo-baselocusleq3}.
Consequently we deduce
 Propositions \ref{prop: r=2 B nondegenerate}
and \ref{prop: r=3 B nondegenerate} by applying the classification 
of smooth varieties of low degree: 
\cite{ionescu-smallinvariants},
\cite{ionescu-smallinvariantsII},
\cite{ionescu-smallinvariantsIII},
\cite{fania-livorni-nine},
\cite{fania-livorni-ten},
\cite{besana-biancofiore-deg11},
\cite{ionescu-degsmallrespectcodim}.
We also apply the 
double point formula in Lemmas:
\ref{lemma: double point formula r=2},
\ref{lemma: double point formula}, 
\ref{lemma: quadric fibration},
\ref{lemma: scroll over surface} and
\ref{lemma: scroll over curve},
 in order to obtain  additional informations on
 $d$ and $\Delta=\deg(\sS)$.  

We summarize our classification 
results in Table \ref{tabella: all cases 3-fold}.
In particular, we provide 
an answer to a question left open
 in the recent preprint \cite{alzati-sierra}. 
\section{Notation and general results}\label{sec: notation}
Throughout the paper we work over $\CC$ and keep the following setting.
\begin{assumption}
Let $\varphi:\PP^n\dashrightarrow\sS:=\overline{\varphi(\PP^n)}\subseteq\PP^{n+a}$ be
a quadratic birational transformation 
with smooth connected base locus $\B$  
and with  $\sS$ nondegenerate, 
linearly normal and factorial.
\end{assumption}
Recall that we can resolve the
 indeterminacies of
 $\varphi$ with the diagram 
\begin{equation}\label{eq: diagram resolving map} 
\xymatrix{ & \widetilde{\PP^n} \ar[dl]_{\pi} \ar[dr]^{\pi'}\\ \PP^n\ar@{-->}[rr]^{\varphi}& & \sS } 
\end{equation}
where $\pi:\widetilde{\PP^n}=\Bl_{\B}(\PP^n)\rightarrow\PP^n$ is the blow-up of
 $\PP^n$ along $\B$ and  
$\pi'=\varphi\circ\pi:\widetilde{\PP^n}\rightarrow\sS$.   
Denote by $\B'$ the base locus of $\varphi^{-1}$, 
 $E$ the exceptional divisor of $\pi$, 
 $E'=\pi'^{-1}(\B')$, 
$H=\pi^{\ast}(H_{\PP^n})$,
$H'={\pi'}^{\ast}(H_{\sS})$, and note that, since
$\pi'|_{\widetilde{\PP^n}\setminus E'}:\widetilde{\PP^n}\setminus E'\rightarrow \sS\setminus\B'$
is an isomorphism,
we have $(\sing(\sS))_{\mathrm{red}}\subseteq (\B')_{\mathrm{red}}$.
We also put 
$r=\dim(\B)$, 
$r'=\dim(\B')$,
$\lambda=\deg(\B)$, 
$g=g(\B)$ the sectional genus of $\B$,
$c_j=c_j(\T_{\B})\cdot H_{\B}^{r-j}$ (resp. $s_j=s_j(\N_{\B,\PP^n})\cdot H_{\B}^{r-j}$) 
the degree of the $j$-th Chern class (resp. Segre class) of $\B$, 
$\Delta=\deg(\sS)$,
$c=c(\sS)$ the \emph{coindex} of $\sS$ 
(the last of which is defined 
by $-K_{\reg(\sS)}\sim (n+1-c)H_{\reg(\sS)}$, 
whenever $\Pic(\sS)=\ZZ\langle H_{\sS}\rangle$).
\begin{assumption}\label{assumption: liftable}
 We suppose that  there exists a rational map 
$\widehat{\varphi}:\PP^{n+a}\dashrightarrow\PP^n$ 
defined by a sublinear system of $|\O_{\PP^{n+a}}(d)|$ and having 
base locus $\widehat{\B}$ 
such that $\varphi^{-1}=\widehat{\varphi}|_{\sS}$ and $\B'=\widehat{\B}\cap\sS$.
We then will say that $\varphi^{-1}$ is \emph{liftable} and 
that $\varphi$ is \emph{of type} $(2,d)$.
\end{assumption}
The above assumption yields the relations:
\begin{equation}\label{eq: lift}
\begin{array}{ll}
 H' \sim  2H-E,   &  H  \sim  dH'-E',  \\
 E'\sim  (2d-1)H-dE, & E \sim  (2d-1)H'-2E' ,
\end{array}
\end{equation}
and hence also
$ \Pic(\widetilde{\PP^n})\simeq \ZZ\langle H \rangle\oplus \ZZ\langle E \rangle
\simeq \ZZ\langle H'\rangle\oplus \ZZ\langle E'\rangle $.
Note that,  by the proofs of
\cite[Proposition~1.3 and 2.1(a)]{ein-shepherdbarron} and
by factoriality of $\sS$, we obtain that $E'$ 
is a reduced and irreducible divisor.
Moreover we have
$\Pic(\sS)\simeq \Pic(\sS\setminus\B')\simeq \Pic(\widetilde{\PP^n}\setminus E')
\simeq \ZZ\langle H'\rangle\simeq \ZZ\langle H_{\sS}\rangle$.
Finally, we require the following:\footnote{See Example 
\ref{example: B2=singSred} and  \cite[Example~4.6]{note} 
 for explicit examples of  special quadratic birational transformations
for which Assumption \ref{assumption: ipotesi} is not satisfied.}
\begin{assumption}\label{assumption: ipotesi}
$(\sing(\sS))_{\mathrm{red}}\neq (\B')_{\mathrm{red}}$.
\end{assumption}
Now we point out that, since $E'$ is irreducible, 
by Assumption \ref{assumption: ipotesi} and \cite[Theorem~1.1]{ein-shepherdbarron},
 we deduce that 
$\pi'|_V:V\rightarrow U$ coincides with the blow-up of $U$ along $Z$,
where $U=\reg(\sS)\setminus\sing((\B')_{\mathrm{red}})$, 
$V=\pi'^{-1}(U)$ and  $Z=U\cap (\B')_{\mathrm{red}}$.
It follows that
$K_{\widetilde{\PP^n}} \sim (-n-1)H+(n-r-1)E \sim (c-n-1)H'+(n-r'-1)E'$, 
from which, together with (\ref{eq: lift}), we obtain
$2r+3-n=n-r'-1$ and 
$c=\left( 1-2d\right) r+dn-3d+2$.
One can also easily see that,
for the general point 
$x\in\Sec(\B)\setminus \B$,  
$\overline{\varphi^{-1}\left(\varphi\left(x\right)\right)}$
is a linear space of dimension $n-r'-1$   and
$\overline{\varphi^{-1}\left(\varphi\left(x\right)\right)}\cap \B$ 
is a quadric hypersurface,  which coincides with 
the entry locus $\Sigma_{x}(\B)$ of $\B$ with respect to $x$.
For more details we refer the reader to
\cite[Proposition~2.3]{ein-shepherdbarron} and \cite[Proposition~3.1]{note}. 
So we can establish one of the main results useful for our purposes:
\begin{proposition}\label{prop: B is QEL}
 $\Sec(\B)\subset\PP^n$ is a hypersurface of degree $2d-1$ and 
       $\B$ is a $QEL$-variety of type $\delta=2r+2-n$.
\end{proposition}
In many cases, $\B$ has a much stronger property of being $QEL$-variety.
Recall that a subscheme $X\subset\PP^n$ is said to have the $K_2$ property if 
$X$ is cut out by quadratic forms $F_0,\ldots,F_N$ 
such that the  Koszul relations among the $F_i$ are generated by linear syzygies.
We have the following fact (see \cite{vermeire} and \cite{alzati-syz}):
\begin{fact}\label{fact: K2 property}
Let $X\subset\PP^n$ be 
 a smooth  variety 
cut out by quadratic forms $F_0,\ldots,F_N$
satisfying $K_2$ property and let 
$F=[F_0,\ldots,F_N]:\PP^n\dashrightarrow\PP^N$ be the 
induced rational map. Then 
for every $x\in\PP^n\setminus X$, 
$\overline{F^{-1}\left(F\left(x\right)\right)}$ is 
a linear space of dimension 
$n+1-\mathrm{rank}\left(\left({\partial F_i}/{\partial x_j}(x)\right)_{i,j}\right)$;
moreover,
$\dim(\overline{F^{-1}\left(F\left(x\right)\right)})>0$ 
if and only if $x\in\Sec(X)\setminus X$ and in this case  
$\overline{F^{-1}\left(F\left(x\right)\right)}\cap X$ is a quadric hypersurface,
which coincides with 
the entry locus $\Sigma_{x}(X)$ of $X$ with respect to $x$.
\end{fact}
We have a simple sufficient 
condition for the $K_2$ property (see \cite[Proposition~2]{alzati-russo-subhomaloidal}):
\begin{fact}\label{fact: test K2}
 Let $X\subset\PP^n$ be a smooth linearly normal variety and suppose $h^1(\O_X)=0$
if $\dim(X)\geq2$. Putting $\lambda=\deg(X)$ and $s=\mathrm{codim}_{\PP^n}(X)$ we have:
\begin{itemize}
\item if $\lambda\leq 2s+1$,  then $X$ is arithmetically Cohen-Macaulay;
\item if $\lambda\leq 2s$,  then the homogeneous ideal of $X$ 
is generated by quadratic forms;
\item if $\lambda\leq2s-1$,  then the syzygies of the generators 
of the homogeneous ideal of $X$ are generated by the linear ones.
\end{itemize}
\end{fact}
\begin{remark}
Let $\psi:\PP^n\dashrightarrow\mathbf{Z}:=\overline{\psi(\PP^n)}\subseteq\PP^{n+a}$ be
a birational transformation ($n\geq3$). 

We point out that, from Grothendieck's Theorem on parafactoriality (Samuel's Conjecture) 
\cite[\Rmnum{11} Corollaire~3.14]{sga2} it follows that $\mathbf{Z}$ is factorial 
whenever it is a local complete intersection 
 with $\dim(\sing(\mathbf{Z}))<\dim(\mathbf{Z})-3$.
Of course, every complete intersection in a smooth variety 
is a local complete intersections.

Moreover, $\psi^{-1}$ is liftable whenever
  $\Pic(\mathbf{Z})=\ZZ\langle H_{\mathbf{Z}}\rangle$ 
and $\mathbf{Z}$ is factorial and projectively normal. 
So, from \cite{larsen-coomology} and \cite[\Rmnum{4} Corollary~3.2]{hartshorne-ample}, 
$\psi^{-1}$ is liftable whenever
$\mathbf{Z}$ is either 
  smooth and projectively normal with $n\geq a+2$ 
or a factorial % of dimension $n\geq3$ 
complete intersection. 
\end{remark}
\section{Numerical restrictions}
Proposition \ref{prop: B is QEL} already provides 
a restriction on the invariants of the transformation $\varphi$;
 here we give further restrictions of this kind.
\begin{proposition}\label{prop: hilbert polynomial}
 Let $\epsilon=0$ if $\langle \B \rangle =\PP^n$ and let  $\epsilon=1$ otherwise.  
\begin{itemize} 
\item   If $r=1$ we have: 
  \begin{eqnarray*}
   \lambda &=&  (n^2-n+2\epsilon-2a-2)/2  ,  \\
         g &=&  (n^2-3n+4\epsilon-2a-2)/2  .
  \end{eqnarray*}
\item If $r=2$ we have: 
 \begin{eqnarray*}
   \chi(\O_{\B}) &=&  (2a-n^2+5n+2g-6\epsilon+4)/4 , \\
         \lambda &=&  (n^2-n+2g+2\epsilon-2a-4)/4 . \\
 \end{eqnarray*}
\item If $r=3$ we have: 
 \begin{eqnarray*}
   \chi(\O_{\B}) &=& (4\lambda-n^2+3n-2g-4\epsilon+2a+6)/2  .
 \end{eqnarray*}
\end{itemize}
\end{proposition}
\begin{proof}
By Proposition \ref{prop: B is QEL} we have
$h^0(\PP^n,\I_{\B}(1))=\epsilon$.
Since $\sS$ is normal and linearly normal, 
we have $h^0(\PP^n,\I_{\B}(2))=n+1+a$ (see \cite[Lemma~2.2]{note}).
Moreover,
 since $n\leq 2r+2$ (being $\delta\geq0$), 
proceeding as in  \cite[Lemma~3.3]{note}
 (or applying
 \cite[Proposition~1.8]{mella-russo-baselocusleq3}), 
 we obtain
$h^j(\PP^n,\I_{\B}(k))=0$ for every $j,k\geq1$.
So we obtain
$\chi(\O_{\B}(1))=n+1-\epsilon$ and 
$\chi(\O_{\B}(2))= (n+1)(n+2)/2 - (n+1+a)$.
\end{proof}
\begin{proposition}\label{prop: segre and chern classes}  \hspace{1pt}
\begin{itemize}
\item  If $r=1$ we have:
\begin{eqnarray*}
{c}_{1} &=& 2-2\,g, \\
{s}_{1} &=& \left( -n-1\right) \,\lambda-2\,g+2, \\
d &=& \left(2\,\lambda-{2}^{n}\right)/\left(\left( 2\,n-2\right) \,\lambda-{2}^{n+1}-4\,g+4\right), \\
\Delta &=& \left( 1-n\right) \,\lambda+{2}^{n}+2\,g-2.
\end{eqnarray*}
\item If $r=2$ we have:
\begin{eqnarray*}
{c}_{1} &=& \lambda-2\,g+2, \\
{c}_{2} &=& -\left(\left( {n}^{2}-3\,n\right) \,\lambda-{2}^{n+1}+\left( 4-4\,g\right) \,n+4\,g+2\,\Delta-4\right)/2, \\
{s}_{1} &=& -n\,\lambda-2\,g+2, \\
{s}_{2} &=& 2\,n\,\lambda+{2}^{n}+\left( 4\,g-4\right) \,n-\Delta, \\
d\,\Delta &=& \left( 2-n\right) \,\lambda+{2}^{n-1}+2\,g-2. \\
\end{eqnarray*}
\item If $r=3$ we have: 
\begin{eqnarray*}
{c}_{1} &=& 2\,\lambda-2\,g+2, \\
{c}_{2} &=& -\left(\left( {n}^{2}-5\,n+2\right) \,\lambda-{2}^{n}+\left( 4-4\,g\right) \,n+12\,g+2\,d\,\Delta-12\right)/2, \\
{c}_{3} &=& \left(\left( 2\,{n}^{3}-12\,{n}^{2}+22\,n-12\right) \,\lambda+9\,{2}^{n}+n\,\left( -3\,{2}^{n}+18\,g+6\,d\,\Delta-18\right) \right.\\
        &&  \left. +\left( 6-6\,g\right) \,{n}^{2}-24\,g+\left( -6\,d-6\right) \,\Delta+24\right)/6, \\
{s}_{1} &=& \left( 1-n\right) \,\lambda-2\,g+2, \\
{s}_{2} &=& \left(\left( 4\,n-4\right) \,\lambda+{2}^{n}+\left( 8\,g-8\right) \,n-8\,g-2\,d\,\Delta+8\right)/2, \\
{s}_{3} &=& \left(\left( 2\,{n}^{3}-12\,{n}^{2}+10\,n\right) \,\lambda+3\,{2}^{n}+n\,\left( -3\,{2}^{n}+12\,g+6\,d\,\Delta-12\right) \right. \\
        && \left. +\left( 12-12\,g\right) \,{n}^{2}-3\,\Delta\right)/3 .
\end{eqnarray*}
\end{itemize}
\end{proposition}
\begin{proof}
See also \cite{crauder-katz-1989} and \cite{crauder-katz-1991}.
By \cite[page~291]{crauder-katz-1989} we see that 
\begin{displaymath}
 H^j\cdot E^{n-j}=
\left\{ 
\begin{array}{ll} 
1, & \mbox{if } j= n ; \\
0, & \mbox{if } r+1\leq j \leq n-1 ; \\
(-1)^{n-j-1} s_{r-j}, & \mbox{if } j \leq r .
\end{array} 
\right.
\end{displaymath}
Since $H'=2H-E$ and $H=dH'-E'$ we have
\begin{eqnarray}
\Delta &=& {H'}^n=(2H-E)^n,  \\
d \Delta &=& d{H'}^{n}={H'}^{n-1}\cdot(dH'-E')=(2H-E)^{n-1}\cdot H.  
\end{eqnarray}
From the exact sequence
$0\rightarrow\mathcal{T}_{\B}\rightarrow \mathcal{T}_{\PP^n}|_{\B}\rightarrow\mathcal{N}_{\B,\PP^n}\rightarrow0$
we get: 
\begin{eqnarray}
s_1 &=& - \lambda \left( n+1\right) + {c}_{1} ,\\
s_2 &=&  \lambda \begin{pmatrix}n+2\cr 2\end{pmatrix}-{c}_{1} \left( n+1\right) +{c}_{2} ,\\ 
s_3 &=& -\lambda \begin{pmatrix}n+3\cr 3\end{pmatrix}+{c}_{1} \begin{pmatrix}n+2\cr 2\end{pmatrix}-{c}_{2} \left( n+1\right) +{c}_{3} ,\\ 
  & \vdots & \nonumber
\end{eqnarray}
Moreover $c_1=-K_{\B}\cdot H_{\B}^{r-1}$ 
and it can be expressed as a function of $\lambda$ and $g$.
Thus we found $r+3$ 
 independent equations on the $2r+5$ variables: 
$c_1,\ldots,c_r,s_1,\ldots,s_r,d,\Delta,\lambda,g,n$.
\end{proof}
\begin{remark}
Proposition \ref{prop: segre and chern classes} 
holds under less restrictive assumptions, as shown in the above proof.
Here we treat the special case:
let $\psi:\PP^8\dashrightarrow\mathbf{Z}:=\overline{\psi(\PP^8)}\subseteq\PP^{8+a}$ be a quadratic 
rational map whose base locus 
is a smooth irreducible $3$-dimensional variety $X$.
Without any other restriction on $\psi$, denoting
with $\pi:\Bl_X(\PP^8)\rightarrow \PP^8$ 
the blow-up of $\PP^8$ along $X$ and with $s_i(X)=s_i(\N_{X,\PP^8})$, we have
\begin{equation}\label{eq: grado mappa razionale}
\deg(\psi)\deg(\mathbf{Z}) = (2\pi^{\ast}(H_{\PP^8})-E_X)^8 =
-s_3(X)-16s_2(X)-112s_1(X)-448\deg(X)+256.
\end{equation}
 Moreover, if $\psi$ is birational with liftable inverse and 
$\dim(\sing(\mathbf{Z}))\leq 6$,
 we also have
\begin{equation}\label{eq: sollevabile}
d \deg(\mathbf{Z}) =(2\pi^{\ast}(H_{\PP^8})-E_X)^7\cdot \pi^{\ast}(H_{\PP^8}) = -s_2(X) -14 s_1(X) -84 \deg(X) +128,
\end{equation}
where $d$ denotes the degree of the linear system defining $\psi^{-1}$.
%
%   Let $f:\widetilde{\PP^n}\stackrel{\pi'}{\rightarrow}\sS\hookrightarrow\PP^{N}$. If $\dim(\B')\leq n-2$, then $f_{\ast}(E')=0$ and hence $$\deg(H'^{n-1}\cdot E')=\deg(f^{\ast}(H_{\PP^N}^{n-1})\cdot E')=\deg(H_{\PP^N}^{n-1}\cdot f_{\ast}(E'))=0.$$ Now, let $(\B')_{\mathrm{red}}=B_1\cup\cdots\cup B_l\cup B_{l+1}\cup\cdots\cup B_t$,  be the decomposition into irreducible components, with $B_1,\ldots,B_l\nsubseteq \sing(\sS)=\sS\setminus U$ and $B_{l+1},\ldots,B_t\subseteq \sing(\sS)=\sS\setminus U$;  $U:=\reg(\sS)$, $g:=\psi^{-1}:\sS\dashrightarrow\PP^n$. We have $\mathrm{Bs}(g|_U)=\mathrm{Bs}(g)\cap U=\B'\cap U =(B_1\cup\cdots\cup B_l)\cap U$ and $\mathrm{codim}_{\sS}(B_i)=\mathrm{codim}_{U}(B_i\cap U)\geq 2$ for $i=1,\ldots,l$.  Thus we have $\mathrm{codim}_{\sS}(\B')\geq2$ if, for example, $\mathrm{codim}_{\sS}(\sing(\sS))\geq 2$. 
%
\end{remark}
Proposition \ref{prop: double point formula}
is a translation of the well-known 
\emph{double point formula} (see for example \cite{peters-simonis} and \cite{laksov}),
taking into account Proposition \ref{prop: B is QEL}.
\begin{proposition}\label{prop: double point formula}
If $\delta=0$ 
then 
$$
2(2d-1)=
\lambda^2 - 
\sum_{j=0}^{r}\begin{pmatrix} 2r+1 \cr j \end{pmatrix} s_{r-j}(\T_{\B})\cdot H_{\B}^{j}.
$$
\end{proposition}
\section{Case of dimension 1}\label{sec: dim 1}
Lemma \ref{lemma: numerical 1-fold}  
directly follows from  
Propositions \ref{prop: hilbert polynomial} 
and \ref{prop: segre and chern classes}.
\begin{lemma}\label{lemma: numerical 1-fold}
If $r=1$, then 
one of the following cases holds:
\begin{enumerate}[(A)]
 \item $n=3$, $a=1$, $\lambda=2$, $g=0$, $d=1$, $\Delta=2$; 
 \item $n=4$, $a=0$, $\lambda=5$, $g=1$, $d=3$, $\Delta=1$; 
 \item $n=4$, $a=1$, $\lambda=4$, $g=0$, $d=2$, $\Delta=2$; 
 \item\label{case: escluso 1-fold} $n=4$, $a=2$, $\lambda=4$, $g=1$, $d=1$, $\Delta=4$; 
 \item $n=4$, $a=3$, $\lambda=3$, $g=0$, $d=1$, $\Delta=5$.
\end{enumerate}
\end{lemma}
\begin{proposition}\label{prop: possibili casi 1-fold}
If $r=1$, then 
one of the following cases holds:
\begin{enumerate}[(I)]
 \item $n=3$, $a=1$,  $\B$ is a conic;
 \item $n=4$, $a=0$,  $\B$ is an elliptic curve of degree $5$;
 \item $n=4$, $a=1$,  $\B$ is the rational normal quartic curve;
 \item $n=4$, $a=3$,  $\B$ is the twisted cubic curve.
\end{enumerate}
\end{proposition}
\begin{proof}
 From Lemma \ref{lemma: numerical 1-fold}
 it remains only to exclude case (\ref{case: escluso 1-fold}).
In this case $\B$ is a complete intersection of two quadrics in $\PP^3$ 
% (see \cite[\Rmnum{4} Exercise~3.6]{hartshorne-ag}) 
and also 
it is an $OADP$-curve. 
This is absurd because the only $OADP$-curve is the twisted cubic curve.
\end{proof}
\section{Case of dimension 2}\label{sec: dim 2}
Proposition \ref{prop: possibili casi 2-fold} follows from 
            \cite[Propositions~1.3 and 3.4]{russo-qel1} and
            \cite[Theorem~4.10]{ciliberto-mella-russo}.
\begin{proposition}\label{prop: possibili casi 2-fold} 
If $r=2$, then either $n=6$, $d\geq2$, $\langle \B \rangle = \PP^6$, or 
one of the following cases holds:
\begin{enumerate}[(I)] 
  \setcounter{enumi}{4} 
 \item\label{case 2-fold a}  $n=4$,  $d=1$, $\delta=2$, 
              $\B=\PP^1\times\PP^1\subset\PP^3\subset\PP^4$;
 \item\label{case 2-fold b}  $n=5$,  $d=1$, $\delta=1$, 
              $\B$ is a hyperplane section of $\PP^1\times\PP^2\subset\PP^5$; 
 \item\label{case 2-fold c}  $n=5$,  $d=2$, $\delta=1$, 
              $\B=\nu_2(\PP^2)\subset\PP^5$ is the Veronese surface;
 \item\label{case 2-fold d}  $n=6$,  $d=1$, $\delta=0$, $\B\subset\PP^5$ is an $OADP$-surface, i.e. $\B$ is as in one of the following cases: 
         \begin{enumerate}[($\ref{case 2-fold d}_1$)]
         \item\label{case 2-fold d1} $\PP_{\PP^1}(\O(1)\oplus\O(3))$ or
                                     $\PP_{\PP^1}(\O(2)\oplus\O(2))$;        
         \item\label{case 2-fold d2} del Pezzo surface of degree  $5$ 
                              (hence the blow-up of  
                 $\PP^2$ at $4$ points 
                 $p_1,\ldots,p_4$ and $|H_{\B}|=|3H_{\PP^2}-p_1-\cdots-p_4|$).
         \end{enumerate}
\end{enumerate}
\end{proposition}
\begin{lemma}\label{lemma: r=2 B nondegenerate}
If $r=2$, $n=6$ and $\langle \B \rangle = \PP^6$, 
then  one of the following cases holds:
\begin{enumerate}[(A)]
\item \label{case 2-fold a=0 lambda=7}
 $a=0$,  $\lambda=7$, $g=1$, $\chi(\O_{\B})=0$;
\item \label{case 2-fold a leq 3}
 $0\leq a \leq 3$, $\lambda=8-a$, $g=3-a$, $\chi(\O_{\B})=1$.
\end{enumerate}
\end{lemma}
\begin{proof}
By Proposition \ref{prop: hilbert polynomial}
 it follows that
$g=2\lambda+a-13$ and $\chi(\O_{\B})=\lambda+a-7$.
By \cite[Lemma~6.1]{note} and using that $g\geq0$
(proceeding as in \cite[Proposition~6.2]{note}),
we obtain $(13-a)/2 \leq \lambda \leq 8-a$.
\end{proof}
\begin{lemma}\label{lemma: double point formula r=2}
 If $r=2$, $n=6$ and $\langle \B \rangle = \PP^6$, then
one of the following cases holds:
\begin{itemize}
 \item $a=0$, $d=4$, $\Delta=1$;
 \item $a=1$, $d=3$, $\Delta=2$;
 \item $a=2$, $d=2$, $\Delta=4$;
 \item $a=3$, $d=2$, $\Delta=5$.
\end{itemize}
\end{lemma}
\begin{proof}
 We have $s_1(\T_{\B})\cdot H_{\B} = -c_1 $ and 
 $ s_2(\T_{\B}) = c_1^2-c_2=12\chi(\O_{\B}) -2c_2 $. 
So, by Proposition \ref{prop: double point formula}, we obtain 
\begin{equation}
2(2d-1) = \lambda^2-10\lambda-12\chi(\O_{\B})+2c_2+5c_1 .
\end{equation}
Now, by Propositions \ref{prop: hilbert polynomial} and
 \ref{prop: segre and chern classes}, 
we obtain
\begin{equation}
d\Delta = 2a+4,\quad  
 \Delta = (g^2+(-2a-4)g-16d+a^2-4a+75)/8 ,
\end{equation}
and then we conclude by Lemma \ref{lemma: r=2 B nondegenerate}.
\end{proof}
\begin{proposition}\label{prop: r=2 B nondegenerate}
If $r=2$,  $n=6$ and $\langle \B\rangle=\PP^6$ 
 then  one of the following cases holds:
\begin{enumerate}[(I)]
\setcounter{enumi}{8} 
\item $a=0$,  $\lambda=7$,  $g=1$, $\B$ is an elliptic scroll $\PP_{C}(\E)$ with $e(\E)=-1$;
\item $a=0$,  $\lambda=8$,  $g=3$, $\B$ is the blow-up of $\PP^2$ at $8$ points $p_1\ldots,p_8$, $|H_{\B}|=|4H_{\PP^2}-p_1-\cdots-p_8|$;
\item $a=1$,  $\lambda=7$,  $g=2$, $\B$ is the blow-up of $\PP^2$ at $6$ points $p_0\ldots,p_5$, $|H_{\B}|=|4H_{\PP^2}-2p_0-p_1-\cdots-p_5|$;
\item $a=2$,  $\lambda=6$,  $g=1$, $\B$ is the blow-up of $\PP^2$ at $3$ points $p_1,p_2,p_3$, $|H_{\B}|=|3H_{\PP^2}-p_1-p_2-p_3|$;
\item $a=3$,  $\lambda=5$,  $g=0$, $\B$ is a rational normal scroll.
\end{enumerate}
\end{proposition}
\begin{proof}
For $a=0$, $a=1$ and $a\in\{2,3\}$ 
the statement follows, respectively, from \cite{crauder-katz-1989},
 \cite[Proposition~6.2]{note} and \cite{ionescu-smallinvariants}. 
\end{proof}
\section{Case of dimension 3}\label{sec: dim 3}
Proposition \ref{prop: possibili casi C1} follows from: 
            \cite[Proposition~1.3 and 3.4]{russo-qel1},
            \cite{fujita-3-fold},
            \cite{ionescu-russo-conicconnected},
            \cite[page~62]{fujita-polarizedvarieties} and
            \cite{ciliberto-mella-russo}.
\begin{proposition}\label{prop: possibili casi C1} 
If $r=3$, then either $n=8$, $d\geq2$, $\langle \B \rangle = \PP^8$, or  
one of the following cases holds:
\begin{enumerate}[(I)]
 \setcounter{enumi}{13} 
 \item\label{case C1 a}  $n=5$, $d=1$, $\delta=3$, $\B=Q^3\subset\PP^4\subset\PP^5$ is a quadric;
 \item\label{case C1 b}  $n=6$, $d=1$, $\delta=2$, $\B=\PP^1\times\PP^2\subset\PP^5\subset\PP^6$;
 \item\label{case C1 c}  $n=7$, $d=1$, $\delta=1$, $\B\subset\PP^6$ is as 
in one of the following cases:
         \begin{enumerate}[($\ref{case C1 c}_1$)]
         \item\label{case C1 c1} $\PP_{\PP^1}(\O(1)\oplus\O(1)\oplus\O(2))$;
         \item\label{case C1 c2} linear section of $\GG(1,4)\subset\PP^9$; 
         \end{enumerate}
 \item\label{case C1 d}  $n=7$, $d=2$, $\delta=1$, $\B$ is a hyperplane section of $\PP^2\times\PP^2\subset\PP^8$;
 \item\label{case C1 e}  $n=8$,  $d=1$, $\delta=0$, $\B\subset\PP^7$ is an $OADP$-variety, 
 i.e. $\B$ is as in one of the following cases:
 \begin{enumerate}[($\ref{case C1 e}_1$)]
 \item\label{case C1 e1} $\PP_{\PP^1}(\O(1)\oplus\O(1)\oplus\O(3))$ or
                         $\PP_{\PP^1}(\O(1)\oplus\O(2)\oplus\O(2))$; 
 \item\label{case C1 e2} Edge variety of degree $6$ (i.e. $\PP^1\times\PP^1\times\PP^1$) or 
                         Edge variety of degree $7$;
 \item\label{case C1 e3} $\PP_{\PP^2}(\E)$, where $\E$ 
is a vector bundle with  $c_1(\E)=4$ and $c_2(\E)=8$, given as an extension by
the following exact sequence
$0\rightarrow\O_{\PP^2}\rightarrow\E\rightarrow 
\I_{\{p_1,\ldots,p_8\},\PP^2}(4)\rightarrow0$. 
 \end{enumerate}
\end{enumerate}
\end{proposition}
In the following we denote by 
$\Lambda\subsetneq C\subsetneq S\subsetneq \B$ 
a sequence of general linear sections of $\B$.
\begin{lemma}\label{lemma: r=3 B nondegenerate}
If $r=3$, $n=8$ and $\langle \B \rangle = \PP^8$,
 then one of the following cases holds:
\begin{enumerate}[(A)]
\item \label{a=0,lambda=13}
 $a=0$, $\lambda=13$, $g=8$, $K_S\cdot H_S=1$, $K_S^2=-1$; 
\item \label{a=1,lambda=12}
 $a=1$, $\lambda=12$, $g=7$, $K_S\cdot H_S=0$, $K_S^2=0$; 
\item \label{a geq2}
 $0\leq a\leq6$, $\lambda=12-a$, $g=6-a$, $K_S\cdot H_S=-2-a$.
\end{enumerate}
\end{lemma}
\begin{proof}
Firstly we note that,
from the exact sequence
 $0\rightarrow\mathcal{T}_{S}\rightarrow\mathcal{T}_{\B}|_{S}\rightarrow\O_{S}(1)\rightarrow0$,
 we deduce 
$c_2=c_2(S)+c_1(S)=12\chi(\O_{S})-K_{S}^2-K_{S}\cdot H_{S}$
and hence
\begin{equation}
K_S^2=14\lambda+12\chi(\O_S)-12g+d\Delta-116 = -22\lambda+12g+d\Delta-12a+184.
\end{equation}
Secondly we note that (see \cite[Lemma~6.1]{note}), putting 
 $h_{\Lambda}(2):=h^0(\PP^5,\O(2))-h^0(\PP^5,\I_{\Lambda}(2))$,
we have  %(see also \cite[page~29]{ciliberto-hilbertfunctions})
%   We have ($c=5$)
%   $$h^0(\PP^c,\I_{\Lambda}(2))=\left\{\begin{array}{ll} h^0(\PP^c,\O(2))-\lambda & \mbox{if }\lambda\leq 2c+1, \\\leq h^0(\PP^c,\O(2))-(2c+1), & \mbox{if } \lambda>2c+1. \end{array} \right.$$
%   Hence
%   $$ h_{\Lambda}(2) \geq \left\{\begin{array}{ll} \lambda & \mbox{if }\lambda\leq 2c+1, \\ 2c+1, & \mbox{if } \lambda>2c+1, \end{array} \right. $$
%   i.e.  $$ h_{\Lambda}(2) \geq \min\{\lambda,2c+1\}. $$
\begin{equation}\label{hilbert-function}
 \mathrm{min}\{\lambda,11\} \leq h_{\Lambda}(2)\leq 
         21-h^0(\PP^8,\I_{\B}(2))=12-a.
\end{equation}
Now we establish the following:
\begin{claim}\label{claim: KsHs<0} 
If $K_S\cdot H_S\leq0$ and $K_S\nsim 0$, then $\lambda=12-a$ and $g=6-a$.
\end{claim}
\begin{proof}[Proof of the Claim]
Similarly to \cite[Case~6.1]{note}, we obtain that
$P_{\B}(-1)=0$ and $P_{\B}(0)=1-q$, where 
$q:=h^1(S,\O_S)=h^1(\B,\O_{\B})$;  in particular $g=-5q-a+6$ and  $\lambda=-3q-a+12$.
Since $g\geq0$  we have $5q\leq 6-a$ and
the possibilities are:
if $a\leq1$ then $q\leq1$; if $a\geq 2$ then $q=0$.
If $(a,q)=(0,1)$ then $(g,\lambda)=(1,9)$ 
and the case is excluded by 
 \cite[Theorem~12.3]{fujita-polarizedvarieties}\footnote{Note that 
$\B$ cannot be a scroll over a curve 
(this follows from (\ref{eq: relation scroll}) and (\ref{eq: second relation scroll}) below
   and also it follows from \cite[Proposition~3.2(i)]{mella-russo-baselocusleq3}).}; 
if $(a,q)=(1,1)$ then $(g,\lambda)=(0,8)$ 
and the case is excluded by 
\cite[Theorem~12.1]{fujita-polarizedvarieties}.
Thus we have $q=0$ and hence
$g=6-a$ and $\lambda=12-a$;
in  particular we have $a\leq 6$.
\end{proof}
Now we discuss the cases according 
to the value of  $a$.
\begin{case}[$a=0$]
It is clear that $\varphi$ must be of type $(2,5)$ and hence 
$K_S^2=-22\lambda+12g+189$. 
%   \color{red}   If $h^1(\O_{\B})=0$, we can use the following argument:   Since $\B$ is cut out by quadrics and    it cannot be a scroll over a curve,    by \cite[Proposition~5.2]{ciliberto-mella-russo} it follows that   the support of the base locus    of the general tangential projection $\tau_{x,\B}:\B\dashrightarrow W_{x,\B}\subset\PP^4$,    consists of $0\leq k< \infty$ lines through  $x$.   Now, just as pointed out in the first proof of Proposition \ref{prop: 3-fold in P8},    we deduce  the relation $\lambda-8+k=\deg(W_{x,\B})$ and also that    $\deg(W_{x,\B})\leq d=5$.    Thus    $\lambda-8+k\leq 5$ and in particular we have  $\lambda\leq13.$  \color{black}
By Claim \ref{claim: KsHs<0}, if $K_S\cdot H_S=2g-2-\lambda<0$,
we fall into case (\ref{a geq2}). 
So we suppose that $K_S\cdot H_S\geq0$, namely that $g\geq\lambda/2+1$.
From Castelnuovo's bound it follows that
 $\lambda\geq12$  and if $\lambda=12$ then  
$K_S\cdot H_S=0$,  $g=7$ and hence $K_S^2=9$.
Since this is impossible by Claim \ref{claim: KsHs<0}, we conclude that
 $\lambda\geq 13$.
Now by (\ref{hilbert-function}) it follows that
 $11\leq h_{\Lambda}(2)\leq12$, but
if $h_{\Lambda}(2)=11$ 
from Castelnuovo Lemma \cite[Lemma~1.10]{ciliberto-hilbertfunctions}  
we obtain a contradiction.
Thus we have $h_{\Lambda}(2)=12$ and
$h^0(\PP^5,\I_{\Lambda}(2))=h^0(\PP^8,\I_{\B}(2))=9$.
So from \cite[Theorem~3.1]{ciliberto-hilbertfunctions}  we deduce 
that $\lambda\leq 14$ and furthermore, 
 by the refinement of Castelnuovo's bound contained in 
\cite[Theorem~2.5]{ciliberto-hilbertfunctions}, 
we obtain $g\leq 2\lambda-18$.
In summary we have the following possibilities:
\begin{enumerate}[(i)]
 \item\label{case: T1} $\lambda=13$, $g=8$, $K_S\cdot H_S=1$, $\chi(\O_S)=2$, $K_S^2=-1$;
 \item\label{case: T2} $\lambda=14$, $g=8$, $K_S\cdot H_S=0$, $\chi(\O_S)=-1$, $K_S^2=-23$;
 \item\label{case: T3} $\lambda=14$, $g=9$, $K_S\cdot H_S=2$, $\chi(\O_S)=1$, $K_S^2=-11$;
 \item\label{case: T4} $\lambda=14$, $g=10$, $K_S\cdot H_S=4$, $\chi(\O_S)=3$, $K_S^2=1$.
\end{enumerate}
Case (\ref{case: T1}) coincides with case (\ref{a=0,lambda=13}). 
Case (\ref{case: T2}) is excluded by Claim \ref{claim: KsHs<0}.
In the circumstances of case (\ref{case: T3}), we have $h^1(S,\O_S)=h^2(S,\O_S)=h^0(S,K_S)$.
If $h^1(S,\O_S)>0$, since 
$(K_{\B}+4H_{\B})\cdot K_S=K_S^2+3 K_S\cdot H_S=-5<0$,
we see that $K_{\B}+4H_{\B}$ is not nef and then 
we obtain a contradiction by \cite{ionescu-adjunction}.
If $h^1(S,\O_S)=0$, then we also have $h^1(\B,\O_{\B})=h^2(\B,\O_{\B})=0$ and hence 
$\chi(\O_{\B})=1-h^3(\B,\O_{\B})\leq 1$, against the fact that $\chi(\O_{\B})=2\lambda-g-17=2$. 
Thus case (\ref{case: T3}) does not occur.
Finally,  in the circumstances of case (\ref{case: T4}), note 
that $h^0(S,K_S)=2+h^1(S,\O_S)\geq2$ and we  write 
$|K_S|=|M|+F$, where $|M|$ is the mobile  part of the linear system $|K_S|$ 
and $F$ is the fixed part. If $M_1=M$ 
 is a general member of $|M|$, there exists $M_2\in|M|$
having no common irreducible components with $M_1$ and 
so $M^2=M_1\cdot M_2=\sum_{p}\left(M_1\cdot M_2\right)_{p}\geq0$;
furthermore, by using Bertini Theorem, we see that
 $\sing(M_1)$ consists of points $p$   such that 
the intersection multiplicity $\left(M_1\cdot M_2\right)_{p}$ of $M_1$ and $M_2$ in $p$ is at least $2$.
By definition, we also have $M\cdot F\geq0$ and so we deduce 
$2p_a(M)-2=M\cdot (M+K_S)= 2 M^2+ M\cdot F\geq 0$, from which 
$p_a(M)\geq 1$ and $p_a(M)=2$ if $F=0$.
On the other hand,  we have $M\cdot H_S\leq K_S\cdot H_S=4$ and, 
since $S$ is cut out by quadrics,
 $M$ does not contain  planar curves of degree $\geq3$.
If  $M\cdot H_S=4$, then  $F=0$,
 $M^2=1$ and $M$ is  a (possibly disconnected) smooth curve;
since $p_a(M)=2$, $M$ is actually disconnected % see \cite[page~315]{hartshorne-ag}
 and so it is a disjoint union  of twisted cubics, conics and lines.  
 But then we obtain  the contradiction that
$p_a(M)=1-\#\{\mbox{connected components of }M\}<0$.
If $M\cdot H_S\leq3$,  then $M$ must be either 
a twisted cubic or a union of conics and lines.  In all these cases
we again obtain the contradiction that 
$p_a(M)=1-\#\{\mbox{connected components of }M\}\leq 0$.
Thus case (\ref{case: T4}) does not occur.
\end{case}
\begin{case}[$a=1$] By \cite[Proposition~6.4]{note}   
we fall into case (\ref{a=1,lambda=12}) or (\ref{a geq2}).
 \end{case}
\begin{case}[$a\geq2$] By (\ref{hilbert-function}) it follows that
$\lambda\leq 10$ and by  Castelnuovo's bound  it follows that  
$K_S\cdot H_S\leq -4<0$. Thus, by Claim \ref{claim: KsHs<0}
we fall into case (\ref{a geq2}).
\end{case}
\end{proof}
Now we apply the double point formula 
(Proposition \ref{prop: double point formula}) 
in order to obtain additional numerical restrictions 
under the hypothesis of Lemma \ref{lemma: r=3 B nondegenerate}. 
\begin{lemma}\label{lemma: double point formula}
If $r=3$, $n=8$ and $\langle \B\rangle=\PP^8$, then
$$
K_{\B}^3=\lambda^2+23\lambda-24g-(7d+1)\Delta-4d+36a-226 .
$$
\end{lemma}
\begin{proof}
We have (see \cite[App. A, Exercise~6.7]{hartshorne-ag}):
\begin{eqnarray*}
 s_1(\T_{\B})\cdot H_{\B}^2 &=& -c_1(\B)\cdot H_{\B}^2=K_{\B}\cdot H_{\B}^2 , \\
 s_2(\T_{\B})\cdot H_{\B} &=& c_1(\B)^2\cdot H_{\B}-c_2(\B)\cdot H_{\B} 
 = K_{\B}^2\cdot H_{\B}-c_2(\B)\cdot H_{\B} \\
 &=& 3K_{\B}\cdot H_{\B}^2-2H_{\B}^3-2c_2(\B)\cdot H_{\B}+12\left(\chi(\O_{\B}(H_{\B}))-\chi(\O_{\B})\right), \\
 s_3(\T_{\B}) &=& -c_1(\B)^3+2c_1(\B)\cdot c_2(\B)-c_3(\B)
=K_{\B}^3+48\chi(\O_{\B})-c_3(\B).
\end{eqnarray*}
Hence, applying the double point formula and using 
 the relations 
$\chi(\O_{\B})=2\lambda-g+a-17$, $\chi(\O_{\B}(H_{\B}))=9$, we obtain:
\begin{eqnarray*}
4d-2 &=& 2\,\deg(\Sec(\B))\\
&=& \deg(\B)^2-s_3(\T_{\B})-7\,s_2(\T_{\B})\cdot H_{\B}-21\,s_1(\T_{\B})\cdot H_{\B}^2-35\,H_{\B}^3 \\
 &=& \deg(\B)^2-21\,\deg(\B)-42\,K_{\B}\cdot H_{\B}^2+14\,c_2(\B)\cdot H_{\B}-K_{\B}^3 \\
 && +c_3(\B)-84\,\chi(\O_{\B}(H_{\B}))+36\,\chi(\O_{\B}) \\
 &=& -K_{\B}^3+\lambda^2+23\lambda-24g-(7d+1)\Delta+36a-228.
\end{eqnarray*} 
\end{proof}
\begin{lemma}\label{lemma: quadric fibration}
If $r=3$, $n=8$, $\langle \B\rangle=\PP^8$ and
 $\B$ is a quadric fibration over a curve,
then  one of the following cases holds: 
\begin{itemize}
\item $a=3$, $\lambda=9$, $g=3$, $d=3$, $\Delta=5$; 
\item $a=4$, $\lambda=8$, $g=2$, $d=2$, $\Delta=10$. 
\end{itemize}
\end{lemma}
\begin{proof}
Denote by $\beta:(\B,H_{\B})\rightarrow (Y,H_Y)$ 
the projection over the curve $Y$
such that $\beta^{\ast}(H_Y)=K_{\B}+2H_{\B}$.
We have 
\begin{eqnarray*}
 0&=&\beta^{\ast}(H_Y)^2\cdot H_{\B} 
  = K_{\B}^2\cdot H_{\B}+4K_{\B}\cdot H_{\B}^2+4H_{\B}^3, \\
 0&=& \beta^{\ast}(H_Y)^3= K_{\B}^3+6K_{\B}^2\cdot H_{\B}+12 K_{\B}\cdot H_{\B}^2+8 H_{\B}^3, \\ 
\chi(\O_{\B}(H_{\B})) &=& \frac{1}{12} K_{\B}^2\cdot H_{\B}-\frac{1}{4}K_{\B}\cdot H_{\B}^2+\frac{1}{6}H_{\B}^3+\frac{1}{12}c_2(\B)\cdot H_{\B}+\chi(\O_{\B}), \\ 
\end{eqnarray*}
from which it follows that 
\begin{eqnarray}
K_{\B}^3 &=& -8\lambda+24g-24, \\
c_2(\B)\cdot H_{\B} 
 &=& -36\lambda+26g-12a+298.
\end{eqnarray}
Hence, by Lemma \ref{lemma: double point formula} and 
Proposition \ref{prop: segre and chern classes}, 
 we obtain
\begin{eqnarray}
 d\Delta&=& 23\lambda-16g+12a-180 , \\ 
\Delta+4d&=&\lambda^2-130\lambda+64g-48a+1058 .
\end{eqnarray}
Now the conclusion follows from Lemma 
\ref{lemma: r=3 B nondegenerate},
by observing that 
the case $a=6$ cannot occur. In fact,
if $a=6$, by \cite{ionescu-smallinvariants} 
it follows that $\B$ is a rational normal scroll 
and by a direct calculation (or by Lemma \ref{lemma: scroll over curve}) we see that 
$d=2$ and $\Delta=14$.
\end{proof}
\begin{lemma}\label{lemma: scroll over surface}
If $r=3$, $n=8$, $\langle \B\rangle=\PP^8$ and
 $\B$ is a scroll over a smooth surface $Y$,
then we have:
\begin{eqnarray*}
c_2\left(Y\right) &=& \left(\left(7d-1\right)\lambda^2+\left(177-679d\right)\lambda+\left(292d-92\right)g-28d^2 \right. \\ 
                  && \left. +\left(5554-252a\right)d+36a-1474\right)/\left(2d+2\right), \\ 
\Delta &=& \left(\lambda^2-107\lambda+48g-4d-36a+878\right)/\left(d+1\right) .
\end{eqnarray*}
\end{lemma}
\begin{proof}
Similarly to Lemma \ref{lemma: quadric fibration}, 
denote by $\beta:(\B,H_{\B})\rightarrow (Y,H_Y)$ 
the projection over the surface  $Y$ such that
 $\beta^{\ast}(H_Y)=K_{\B}+2H_{\B}$. Since $\beta^{\ast}(H_Y)^3=0$ we obtain
\begin{eqnarray*}
K_{\B}^3 &=&-8H_{\B}^3-12K_{\B}\cdot H_{\B}^2-6K_{\B}^2\cdot H_{\B} \\
 &=& -30K_{\B}\cdot H_{\B}^2+4H_{\B}^3+6c_2(\B)\cdot H_{\B}-72\chi(\O_{\B}(H_{\B}))+72\chi(\O_{\B}) \\
&=& 130\lambda-72g-6d\Delta+72a-1104.
\end{eqnarray*}
Now  we conclude comparing the last formula
with Lemma \ref{lemma: double point formula} and
using the relation
\begin{equation}
70\lambda-44g+(7d-1)\Delta-596=c_3(\B)=c_1(\PP^1)c_2(Y)=2c_2(Y).
\end{equation}
\end{proof}
\begin{lemma}\label{lemma: scroll over curve}
If $r=3$, $n=8$, $\langle \B\rangle=\PP^8$ and
 $\B$ is a scroll over a smooth curve,
then we have:
 $a=6$, $\lambda=6$, $g=0$, $d=2$, $\Delta=14$.
\end{lemma}
\begin{proof}
We have a projection $\beta:(\B,H_{\B})\rightarrow (Y,H_Y)$ 
 over a curve  $Y$ such that
 $\beta^{\ast}(H_Y)=K_{\B}+3H_{\B}$. By expanding the expressions 
$\beta^{\ast}(H_Y)^2\cdot H_{\B}=0$ and  $\beta^{\ast}(H_Y)^3=0$ we obtain 
$K_{\B}^2\cdot H_{\B}=3\lambda-12g+12$ and 
$K_{\B}^3=54(g-1)$,
and hence by Lemma \ref{lemma: double point formula} we get
\begin{equation}\label{eq: relation scroll}
 \lambda^2+23\lambda-78g-(7d+1)\Delta-4d+36a-172 = 0.
\end{equation}
Also, by expanding the expression $\chi(\O_{\B}(H_{\B}))=9$ we obtain 
$c_2=-35\lambda+30g-12a+294 $
and hence by Proposition \ref{prop: segre and chern classes} we get
\begin{equation}\label{eq: second relation scroll}
 22\lambda-20g-d\Delta+12a-176 = 0.
\end{equation}
Now the conclusion follows from Lemma \ref{lemma: r=3 B nondegenerate}. 
\end{proof}
Finally we conclude our discussion about classification with the following:
\begin{proposition}\label{prop: r=3 B nondegenerate}
If $r=3$, $n=8$ and $\langle \B\rangle=\PP^8$, 
then  one of the following cases holds:
\begin{enumerate}[(I)]
\setcounter{enumi}{18} 
\item $a=0$,  $\lambda=12$,  $g=6$,  $\B$ is a scroll $\PP_{Y}(\E)$ over a birationally ruled surface $Y$ with $K_Y^2=5$, $c_2(\E)=8$ and $c_1^2(\E)=20$; 
\item $a=0$,  $\lambda=13$,  $g=8$,  $\B$ is  obtained as the blow-up of a Fano variety $X$ at a point $p\in X$,  $|H_{\B}|=|H_{X}-p|$;  
\item\label{case: cubic hypersurface} $a=1$,  $\lambda=11$,  $g=5$,  $\B$ is the blow-up of $Q^3$ at $5$ points $p_1,\ldots,p_5$, $|H_{\B}|=|2H_{Q^3}-p_1-\cdots-p_5|$;
\item $a=1$,  $\lambda=11$,  $g=5$,  $\B$ is a scroll over $\PP_{\PP^1}(\O\oplus\O(-1))$; 
\item $a=1$,  $\lambda=12$,  $g=7$,  $\B$ is a linear section of $S^{10}\subset\PP^{15}$;
\item\label{case: scroll over Q2 or quadric fibration} $a=2$,  $\lambda=10$,  $g=4$,  $\B$ is a scroll over $Q^2$;
\item $a=3$,  $\lambda=9$,   $g=3$,  $\B$ is a scroll over $\PP^2$ or a quadric fibration over $\PP^1$; 
\item $a=4$,  $\lambda=8$,   $g=2$,  $\B$ is a hyperplane section of $\PP^1\times Q^3$; 
\item $a=6$,  $\lambda=6$,   $g=0$,  $\B$ is a rational normal scroll.
\end{enumerate}
\end{proposition}
\begin{proof}
For $a=6$ the statement follows from \cite{ionescu-smallinvariants}.
The case with $a=5$ is excluded by  \cite{ionescu-smallinvariants} and Example
\ref{example: a=5}.
For $a=4$ the statement follows from \cite{ionescu-smallinvariantsIII}.  
 For $a\in\{2,3\}$,
by \cite{fania-livorni-nine}, 
   \cite{fania-livorni-ten} and
   \cite{ionescu-smallinvariantsII} 
it follows that the abstract structure of
$\B$  is as asserted, or $a=2$ and $\B$ is a quadric 
fibration over $\PP^1$; the last case is 
excluded by  
Lemma \ref{lemma: quadric fibration}.
 For $a=1$  the statement is just \cite[Proposition~6.6]{note}. 
Now we treat the cases with $a=0$.
\begin{case}[$a=0, \lambda=12$]
 Since $\deg(\B)\leq 2\mathrm{codim}_{\PP^8}(\B)+2$, it follows that 
$(\B,H_{\B})$  must be as in one of the cases
 (a),\ldots,(h) of \cite[Theorem~1]{ionescu-degsmallrespectcodim}.
Cases (a), (d), (e), (g), (h) are of course impossible 
and case (c) is excluded by Lemma \ref{lemma: quadric fibration}. 
% and also because  $k=\#(\L_{x,\B})\leq1$. 
If $\B$ is as in case (b), 
by Lemma \ref{lemma: scroll over curve}
we obtain that
$\B$ is a scroll over a birationally ruled surface. 
Now suppose that   $(\B,H_{\B})$ is as in case (f). 
Thus there is a reduction $(X,H_X)$ as in one of the cases:
\begin{enumerate}[(f1)]
 \item\label{case1} $X=\PP^3$, $H_X\in|\O(3)|$;
 \item\label{case2} $X=Q^3$, $H_X\in|\O(2)|$;
 \item\label{case3} $X$ is a $\PP^2$-bundle over a smooth curve such that 
                    $\O_X(H_X)$ induces $\O(2)$ on each fiber.
\end{enumerate}
By definition of reduction we have $X\subset\PP^{N}$,
where $N=8+s$, 
$\deg(X)=\lambda+s=12+s$ and
 $s$ is the number of points blown up on $X$ to get $\B$. 
Case (f\ref{case1}) and (f\ref{case2}) 
are impossible because
they force $\lambda$ to be
respectively $16$ and $11$.
In case (f\ref{case3}), 
we have a projection $\beta:(X,H_{X})\rightarrow (Y,H_Y)$ 
 over a curve  $Y$ such that
 $\beta^{\ast}(H_Y)=2K_{X}+3H_{X}$. % (see \cite[Lemma~8.1]{besana-biancofiore-numerical}) 
Hence we get 
\begin{displaymath}
 K_X H_X^2= (2K_X+3H_X)^2\cdot H_X/12 -K_X^2\cdot H_X/3 - 3H_X^3/4 = -K_X^2\cdot H_X/3 - 3H_X^3/4 ,
\end{displaymath}
from which we deduce that
\begin{eqnarray*}
 0&=&(2K_X+3H_X)^3 
   = 8K_X^3+36K_X^2\cdot H_X+54K_X\cdot H_X^2+27H_X^3 \\
  &=& 8K_X^3+18K_X^2\cdot H_X-27 H_X^3/2 \\
 &=& 8( K_{\B}^3 - 8s )+18K_X^2\cdot H_X-27 (\deg(\B)+s)/2 \\ 
 &=& 18 K_X^2\cdot H_X-155s/2-210.
\end{eqnarray*} 
Since $s\leq 12$ (see \cite[Lemma~8.1]{besana-biancofiore-numerical}),  
we conclude that case (f)
does not occur. 
Thus,  $\B=\PP_{Y}(\E)$ is a scroll over a surface $Y$; 
moreover,
by Lemma \ref{lemma: scroll over surface} and 
\cite[Theorem~11.1.2]{beltrametti-sommese}, 
% \cite[remark~5.1]{besana-biancofiore-numerical}, 
we obtain $K_Y^2=5$, $c_2(\E)=K_Y^2-K_S^2=8$ and $c_1^2(\E)=\lambda+c_2(\E)=20$.
\end{case}
\begin{case}[$a=0, \lambda=13$]
The proof is located in  \cite[page~16]{mella-russo-baselocusleq3},
but we sketch it for the reader's convenience.
By Lemma \ref{lemma: r=3 B nondegenerate} we know that
$\chi(\O_S)=2$ and  $K_S$ is an exceptional curve of the first kind.
Thus, if we blow-down the divisor $K_S$,  we obtain 
a $K3$-surface. % (see \cite[pages~387,414,422]{hartshorne-ag}, \cite[page~582]{griffiths-harris})
By using adjunction theory 
(see for instance \cite{beltrametti-sommese} or  Ionescu's papers cited in the references)   
and  by Lemmas \ref{lemma: quadric fibration}, 
\ref{lemma: scroll over surface} and \ref{lemma: scroll over curve} it follows that 
the adjunction map $\phi_{|K_{\B}+2H_{\B}|}$ is a generically finite morphism;
moreover, since $(K_{\B}+2H_{\B})\cdot K_S=0$, we see that
$\phi_{|K_{\B}+2H_{\B}|}$ is not a finite morphism.
%  $K_{\B}+2H_{\B}$ is not ample. % (we have $h^0(S,K_S)>0$, $K_S\cdot H_S=1$, hence $K_S$ is effective and integral. Now see \cite[page~14]{debarre}.)
So, we deduce that 
there is a $(\PP^2,\O_{\PP^2}(-1))$ inside $\B$  
% see \cite[page~342, Proof of Lemma~1.2]{ionescu-degsmallrespectcodim} and \cite[Theorems~2.7, 2.8, 2.9]{fania-livorni-ten}
and, after the blow-down of this divisor, we get a smooth Fano $3$-fold 
$X\subset\PP^9$  of sectional genus $8$ and degree $14$. 
\end{case}
\end{proof}
\section{Examples}\label{sec: examples}
 The calculations in the following examples can be verified 
 with the aid of  the computer algebra system \cite{macaulay2}.
\begin{example}[$r=1,2,3; n=3,4,5; a=1; d=1$]\label{example: 1}
See also \cite[\S 2]{note}.
 If $Q\subset\PP^{n-1}\subset\PP^n$ is a smooth quadric, then
the linear system $|\I_{Q,\PP^n}(2)|$ defines  a birational transformation
$\psi:\PP^n\dashrightarrow\sS\subset\PP^{n+1}$ of type $(2,1)$ whose image
 is a smooth quadric.
\end{example}
\begin{example}[$r=1; n=4; a=0; d=3$]\label{example: 2}
See also \cite{crauder-katz-1989}. If $X\subset\PP^4$ 
is a nondegenerate curve of genus $1$ and degree $5$, then $X$ is 
the scheme-theoretic intersection of the quadrics (of rank $3$) containing $X$ and 
$|\I_{X,\PP^4}(2)|$ defines a Cremona transformation 
$\PP^4\dashrightarrow\PP^4$ of type $(2,3)$.
\end{example}
\begin{example}[$r=1,2,3; n=4,5,7; a=1,0,1; d=2$]\label{example: 3}
See also \cite{ein-shepherdbarron} and \cite[Example~4.1]{note}.
If $X\subset\PP^n$ is a Severi variety, then $|\I_{X,\PP^n}(2)|$ defines
a birational transformation
 $\psi:\PP^n\dashrightarrow\PP^n$ of type $(2,2)$ whose  base locus is $X$.
The restriction of  $\psi$
to a general hyperplane is a birational transformation 
 $\PP^{n-1}\dashrightarrow\sS\subset\PP^n$ of type $(2,2)$ and
 $\sS$ is a smooth quadric.
\end{example}
\begin{example}[$r=1; n=4; a=2; d=1$ - not satisfying \ref{assumption: ipotesi}]\label{example: B2=singSred}
We have a special birational transformation
$\psi:\PP^4\dashrightarrow\sS\subset\PP^6$ of type $(2,1)$
with base locus $X$, image  $\sS$
and base locus of the inverse $Y$, as follows:
\begin{eqnarray*}
X &=& V(x_0x_1-x_2^2-x_3^2,-x_0^2-x_1^2+x_2x_3,x_4), \\
\sS &= & V(y_2y_3-y_4^2-y_5^2-y_0y_6,y_2^2+y_3^2-y_4y_5+y_1y_6), \\
 P_{\sS}(t) &=& (4t^4+24t^3+56t^2+60t+24)/4!, \\ 
\sing(\sS) &=& V(y_6,y_5^2,y_4y_5,y_3y_5,y_2y_5,y_4^2,y_3y_4,y_2y_4,2y_1y_4+y_0y_5, \\
           && y_0y_4+2y_1y_5,y_3^2,y_2y_3,y_2^2,y_1y_2+2y_0y_3,2y_0y_2+y_1y_3), \\
P_{\sing(\sS)}(t) &=& t + 5, \\
(\sing(\sS))_{\mathrm{red}} &=& V(y_6,y_5,y_4,y_3,y_2), \\
Y=(Y)_{\mathrm{red}}&=&(\sing(\sS))_{\mathrm{red}}= V(y_6,y_5,y_4,y_3,y_2).
\end{eqnarray*}
See also \cite[Example~4.6]{note} for another example 
in which \ref{assumption: ipotesi} is not satisfied.
\end{example}
\begin{example}[$r=1,2,3; n=4,5,6; a=3; d=1$]\label{example: 5}
See also \cite{russo-simis} and \cite{semple}.
If $X=\PP^1\times\PP^2\subset\PP^5\subset\PP^{6}$, 
then $|\I_{X,\PP^6}(2)|$ defines 
a birational transformation 
$\psi:\PP^{6}\dashrightarrow \sS\subset\PP^{9}$ of type $(2,1)$
whose base locus is $X$ and whose image is $\sS=\GG(1,4)$.
Restricting $\psi$ to a general $\PP^5\subset\PP^{6}$ 
(resp. $\PP^4\subset\PP^{6}$)
we obtain a birational transformation
 $\PP^5\dashrightarrow\sS\subset\PP^{8}$
(resp. $\PP^4\dashrightarrow\sS\subset\PP^{7}$) 
whose  image is a smooth linear section of $\GG(1,4)\subset\PP^{9}$.
\end{example}
\begin{example}[$r=2; n=6; a=0; d=4$]\label{example: 6}
See also \cite{crauder-katz-1989} and \cite{hulek-katz-schreyer}. 
Let $Z=\{p_1,\ldots,p_8\}\subset\PP^2$ be such that 
no $4$ of the $p_i$ are collinear and no $7$ of the $p_i$ lie on a conic and
consider the blow-up $X=\Bl_Z(\PP^2)$ 
embedded in $\PP^6$ by $|4H_{\PP^2}-p_1-\cdots-p_8|$. 
Then the homogeneous ideal of $X$ is
generated by quadrics and $|\I_{X,\PP^6}(2)|$ 
defines a Cremona transformation $\PP^6\dashrightarrow\PP^6$ of type $(2,4)$.
The same happens when
 $X\subset\PP^6$ is a septic elliptic scroll with $e=-1$.
\end{example}
\begin{example}[$r=2; n=6; a=1; d=3$]\label{example: 7}
See also \cite[Examples~4.2 and 4.3]{note}.
If $X\subset\PP^6$ is a general hyperplane section of an
Edge variety  of dimension $3$ and degree $7$ in $\PP^7$,
then $|\I_{X,\PP^6}(2)|$ defines a birational transformation
$\psi:\PP^6\dashrightarrow\sS\subset\PP^7$ of type $(2,3)$ whose
base locus is $X$ and whose image is a rank $6$ quadric.
\end{example}
\begin{example}[$r=2; n=6;a=2;d=2$]\label{example: 8}
If $X\subset\PP^6$ 
is the blow-up of $\PP^2$ at $3$ general points $p_1,p_2,p_3$ with
 $|H_{X}|=|3H_{\PP^2}-p_1-p_2-p_3|$, then 
$\Sec(X)$ is a cubic hypersurface. 
By Fact \ref{fact: K2 property} and \ref{fact: test K2} we deduce  
that $|\I_{X,\PP^6}(2)|$ defines a birational transformation 
$\psi:\PP^6\dashrightarrow\sS\subset\PP^8$ and its type is $(2,2)$.
The image $\sS$ is a complete intersection of two quadrics, $\dim(\sing(\sS))=1$ and 
the base locus of the inverse is $\PP^2\times\PP^2\subset\PP^8$.
Alternatively, we can obtain the transformation 
$\psi:\PP^6\dashrightarrow\sS\subset\PP^8$
by restriction to a general $\PP^6\subset\PP^8$
of the special Cremona transformation $\PP^8\dashrightarrow\PP^8$ of type $(2,2)$.
\end{example}
\begin{example}[$r=2; n=6; a=3; d=2$]\label{example: 9}
See also \cite{russo-simis} and \cite{semple}.
If $X=\PP_{\PP^1}(\O(1)\oplus\O(4))$ or $X=\PP_{\PP^1}(\O(2)\oplus\O(3))$, then
$|\I_{X,\PP^6}(2)|$ defines a birational
 transformations 
 $\psi:\PP^6\dashrightarrow\sS\subset\PP^9$
of type $(2,2)$ whose base locus is $X$ and whose image is
$\sS=\GG(1,4)$.
\end{example}
\begin{example}[$r=2,3;n=6,7;a=5; d=1$]\label{example: 10}
See also \cite[\Rmnum{3} Theorem~3.8]{zak-tangent}.
If $X=\GG(1,4)\subset\PP^9\subset\PP^{10}$, then 
$|\I_{X,\PP^{10}}(2)|$ defines a birational transformation 
$\psi:\PP^{10}\dashrightarrow \sS\subset\PP^{15}$ of type $(2,1)$
whose base locus is $X$ and whose image is the spinorial variety $\sS=S^{10}\subset\PP^{15}$.
Restricting $\psi$
  to a general 
 $\PP^7\subset\PP^{10}$ 
 (resp. $\PP^6\subset\PP^{10}$) 
we obtain a special birational transformation
 $\PP^7\dashrightarrow\sS\subset\PP^{12}$
(resp. $\PP^6\dashrightarrow\sS\subset\PP^{11}$) 
whose dimension of the base locus is
$r=3$ (resp. $r=2$) 
and whose image is a linear section of $S^{10}\subset\PP^{15}$.
In the first case $\sS=\overline{\psi(\PP^7)}$ is smooth while in the second case 
the singular locus of $\sS=\overline{\psi(\PP^6)}$ consists of 
$5$ lines, image of the $5$ Segre $3$-folds 
containing del Pezzo surface of degree $5$ 
and spanned by its pencils of conics.
\end{example}
\begin{example}[$r=2,3; n=6,7; a=6; d=1$]\label{example: 11}
See also \cite{russo-simis}, \cite{semple} and \cite[\Rmnum{3} Theorem~3.8]{zak-tangent}.
We have a birational transformation
$\psi:\PP^{8}\dashrightarrow\GG(1,5)\subset\PP^{14}$ of type $(2,1)$
whose base locus 
is $\PP^1\times\PP^3\subset\PP^7\subset\PP^{8}$ and whose image is $\GG(1,5)$.
Restricting $\psi$
  to a general  $\PP^7\subset\PP^{8}$ 
 we obtain a birational transformation 
 $\PP^7\dashrightarrow\sS\subset\PP^{13}$
whose base locus $X$ is
a rational normal scroll 
and whose image $\sS$ is  a smooth linear section of
 $\GG(1,5)\subset\PP^{14}$. 
Restricting $\psi$
  to a general  $\PP^6\subset\PP^{8}$ 
 we obtain a birational transformation 
 $\psi=\psi|_{\PP^6}:\PP^6\dashrightarrow\sS\subset\PP^{12}$
whose base locus $X$ is
a rational normal scroll 
(hence either $X=\PP_{\PP^1}(\O(1)\oplus\O(3))$ 
or  $X=\PP_{\PP^1}(\O(2)\oplus\O(2))$)
and whose image $\sS$ is  a singular linear section of
 $\GG(1,5)\subset\PP^{14}$. In this case,
we denote by $Y\subset\sS$ the base locus of the inverse of $\psi$ and by 
$F=(F_0,\ldots,F_5):\PP^5\dashrightarrow\PP^5$ 
the restriction of $\psi$ to $\PP^5=\Sec(X)$. 
We have
\begin{eqnarray*}
Y&=&\overline{\psi(\PP^5)}=\overline{F(\PP^5)}=\GG(1,3)\subset\PP^5\subset\PP^{12} , \\
J_4&:=&\left\{x=[x_0,...,x_5]\in\PP^5\setminus X: \mathrm{rank}\left(\left({\partial F_i}/{\partial x_j}(x)\right)_{i,j}\right)\leq 4   \right\}_{\mathrm{red}}\\
    &=& \left\{x=[x_0,...,x_5]\in\PP^5\setminus X: \dim\left(\overline{F^{-1}\left(F(x)\right)}\right)\geq2   \right\}_{\mathrm{red}}\mbox{ and }\dim\left(J_4\right) = 3,\\
\overline{\psi\left(J_4\right)} &=& \left(\sing\left(\sS\right)\right)_{\mathrm{red}} =\PP_{\PP^1}(\O(2)) \subset Y. \\
\end{eqnarray*}
\end{example}
\begin{example}[$r=3;n=8;a=0;d=5$]\label{example: 12}
 See also \cite{hulek-katz-schreyer}. If $\mathcal{X}\subset\PP^9$ is a 
general $3$-dimensional 
 linear section of $\GG(1,5)\subset \PP^{14}$,
 $p\in \mathcal{X}$ is a general point and 
 $X\subset\PP^8$ is the image of $\mathcal{X}$ under the projection from $p$,
then the homogeneous ideal of $X$ is generated by quadrics 
and $|\I_{X,\PP^8}(2)|$ defines a Cremona transformation $\PP^8\dashrightarrow\PP^8$
of type $(2,5)$.
\end{example}
\begin{example}[$r=3;n=8;a=1;d=3$]\label{example: 13} 
See also \cite[Example~4.5]{note}.
If $X\subset\PP^8$ 
is the blow-up 
of the smooth quadric $Q^3\subset\PP^4$ 
at $5$ general points $p_1,\ldots,p_5$ with
 $|H_{X}|=|2H_{Q^3}-p_1-\cdots-p_5|$, then 
$|\I_{X,\PP^8}(2)|$ defines a birational transformation 
$\psi:\PP^8\dashrightarrow\sS\subset\PP^9$ of type $(2,3)$ whose 
base locus is $X$ and whose image 
is a cubic hypersurface with singular locus of dimension $3$.
\end{example}
\begin{example}[$r=3;n=8;a=1;d=4$ - incomplete]\label{example: 14}
By \cite{alzati-fania-ruled} (see also \cite{besana-fania-flamini-f1}) there exists
a smooth irreducible nondegenerate linearly normal $3$-dimensional variety 
 $X\subset\PP^8$  with
 $h^1(X,\O_X)=0$, % vedi pag 15 di \cite{besana-fania-flamini-f1}  
 degree $\lambda=11$, sectional genus $g=5$, having the structure of a 
 scroll $\PP_{\FF^1}(\E)$  with $c_1(\E)=3C_0+5f$ and $c_2(\E)=10$ 
and hence having degrees of the Segre classes 
 $s_1(X)=-85$, $s_2(X)=386$, $s_3(X)=-1330$.
Now, by Fact \ref{fact: test K2}, $X\subset\PP^8$ is 
arithmetically Cohen-Macaulay
and by Riemann-Roch, denoting with $C$ a general curve section of $X$,
 we obtain
\begin{equation}\label{eq: riemann-rock}
h^0(\PP^8,\I_X(2))=h^0(\PP^6,\I_C(2))
= h^0(\PP^6,\O_{\PP^6}(2))-h^0(C,\O_C(2))
% see [Hartshorne, example 1.3.4, page 296]
=28-(2\lambda+1-g),
\end{equation}
hence $h^0(\PP^8,\I_X(2))=10$.
If the homogeneous ideal of $X$ is generated by quadratic forms or at least if
$X=V(H^0(\I_X(2)))$,
the linear system $|\I_X(2)|$ defines a rational map 
$\psi:\PP^8\dashrightarrow\sS=\overline{\psi(\PP^8)}\subset\PP^{9}$ 
whose base locus is
 $X$ and whose image $\sS$ is nondegenerate.
Now, by (\ref{eq: grado mappa razionale}) we deduce   $\deg(\psi)\deg(\sS)=2$, 
from which 
$\deg(\psi)=1$ and $\deg(\sS)=2$.
\end{example}   
\begin{example}[$r=3;n=8;a=1;d=4$]\label{example: 15} 
See also \cite[\S 4]{ein-shepherdbarron} and \cite[Example~4.4]{note}.
If $X\subset\PP^8$ is
a general linear $3$-dimensional section 
of the spinorial variety $S^{10}\subset\PP^{15}$, then
$|\I_{X,\PP^8}(2)|$ defines a
 birational transformation 
 $\psi:\PP^8\dashrightarrow\sS\subset\PP^9$
of type $(2,4)$ 
whose base locus is $X$ and whose image is a smooth quadric.
\end{example}
\begin{example}[$r=3;n=8;a=2;d=3$]\label{example: 16}
By \cite{fania-livorni-ten} (see also \cite{besana-fania-threefolds}) there exists
a smooth irreducible nondegenerate linearly normal $3$-dimensional variety 
 $X\subset\PP^8$  with $h^1(X,\O_X)=0$,  
 degree $\lambda=10$, sectional genus $g=4$, 
having the structure of a  scroll 
$\PP_{Q^2}(\E)$ with $c_1(\E)=\O_Q(3,3)$ and $c_2(\E)=8$ 
and hence having degrees of the Segre classes 
$s_1(X)=-76$, $s_2(X)=340$, $s_3(X)=-1156$.
By Fact \ref{fact: test K2}, $X\subset\PP^8$ is
arithmetically Cohen-Macaulay and its  homogeneous ideal is generated by quadratic forms. 
So by (\ref{eq: riemann-rock}) we have $h^0(\PP^8,\I_X(2))=11$ and 
the linear system $|\I_X(2)|$ defines a rational map 
$\psi:\PP^8\dashrightarrow\sS\subset\PP^{10}$ 
whose base locus is
 $X$ and whose image $\sS$ is nondegenerate.
By (\ref{eq: grado mappa razionale}) it follows that $\deg(\psi)\deg(\sS)=4$ 
and hence $\deg(\psi)=1$ and $\deg(\sS)=4$.
\end{example}
\begin{example}[$r=3;n=8;a=3;d=2,3$]\label{example: 17}
By \cite{fania-livorni-nine} (see also \cite{besana-fania-threefolds}) there exists
a smooth irreducible nondegenerate linearly normal $3$-dimensional variety 
 $X\subset\PP^8$  with $h^1(X,\O_X)=0$,  
 degree $\lambda=9$, sectional genus $g=3$, 
having the structure of a  scroll $\PP_{\PP^2}(\E)$ with $c_1(\E)=4$ and $c_2(\E)=7$ 
(resp. of a quadric fibration over $\PP^1$) 
and hence having degrees of the Segre classes 
$s_1(X)=-67$, $s_2(X)=294$, $s_3(X)=-984$ (resp. 
$s_1(X)=-67$, $s_2(X)=295$, $s_3(X)=-997$).
By Fact \ref{fact: test K2}, $X\subset\PP^8$ is
arithmetically Cohen-Macaulay and its homogeneous ideal is generated by quadratic forms. 
So by (\ref{eq: riemann-rock}) we have $h^0(\PP^8,\I_X(2))=12$ and 
the linear system $|\I_X(2)|$ defines a rational map 
$\psi:\PP^8\dashrightarrow\sS\subset\PP^{11}$ 
whose base locus is
 $X$  and whose image $\sS$ is nondegenerate.
By (\ref{eq: grado mappa razionale}) it follows that 
$\deg(\psi)\deg(\sS)=8$ (resp. $\deg(\psi)\deg(\sS)=5$) and 
in particular 
$\deg(\psi)\neq 0$ i.e. 
$\psi:\PP^8\dashrightarrow\sS$ is 
generically quasi-finite.
Again by Fact \ref{fact: test K2} and Fact \ref{fact: K2 property} it follows
that $\psi$ is birational and hence $\deg(\sS)=8$ (resp. $\deg(\sS)=5$).
\end{example}
\begin{example}[$r=3;n=8;a=4;d=2$]\label{example: 18}
Consider the composition
$$
f:\PP^1\times\PP^3\longrightarrow\PP^1\times Q^3\subset \PP^1\times\PP^4\longrightarrow\PP^9,
$$
where the first map is the identity of $\PP^1$ 
multiplied by
$[z_0,z_1,z_2,z_3]\mapsto [z_0^2,z_0z_1,z_0z_2,z_0z_3,z_1^2+z_2^2+z_3^2]$,
and the last map is
$([t_0,t_1],[y_0,\ldots,y_4])\mapsto [t_0y_0,\ldots,t_0y_4,t_1y_0,\ldots,t_1y_4]
=[x_0,\ldots,x_9]$.
In the equations defining 
$\overline{f(\PP^1\times\PP^3)}\subset\PP^{9}$, 
by replacing 
 $x_9$ with $x_0$, we obtain the quadrics:
\begin{equation}\label{equazioni-di-B-a=4}
 \begin{array}{c}
-x_0x_3 + x_4x_8, \ 
-x_0x_2 + x_4x_7, \ 
 x_3x_7 - x_2x_8, \ 
-x_0x_5 + x_6^2 + x_7^2 + x_8^2, \ 
-x_0x_1 + x_4x_6, \\ 
x_3x_6 - x_1x_8, \ 
x_2x_6 - x_1x_7, \
-x_0^2 + x_1x_6 + x_2x_7 + x_3x_8, \ 
-x_0^2 + x_4x_5, \ 
x_3x_5 - x_0x_8, \\ 
x_2x_5 - x_0x_7, \ 
x_1x_5 - x_0x_6, \ 
x_1^2 + x_2^2 + x_3^2 - x_0x_4.
\end{array}
\end{equation}
Denoting with $I$ the ideal generated by
quadrics (\ref{equazioni-di-B-a=4}) and $X=V(I)$, 
we have that $I$ is saturated (in particular $I_2=H^0(\I_{X,\PP^8}(2))$) 
and $X$ is smooth.
The linear system $|\I_{X,\PP^8}(2)|$ defines 
a birational transformation 
 $\psi:\PP^8\dashrightarrow\sS\subset \PP^{12}$ 
whose base locus is $X$ and whose image is the variety $\sS$ 
with homogeneous ideal generated by:
\begin{equation}
\begin{array}{c}
 y_6y_9-y_5y_{10}+y_2y_{11}, \ 
y_6y_8-y_4y_{10}+y_1y_{11}, \ 
y_5y_8-y_4y_9+y_0y_{11}, \ 
y_2y_8-y_1y_9+y_0y_{10}, \\ 
y_2y_4-y_1y_5+y_0y_6, \ 
y_2^2+y_5^2+y_6^2+y_7^2-y_7y_8+y_0y_9+y_1y_{10}+y_4y_{11}-y_3y_{12}.
\end{array}
\end{equation}
We have $\deg(\sS)=10$ and  $\dim(\sing(\sS))=3$. 
The inverse of $\psi:\PP^8\dashrightarrow\sS$ 
is defined by:
\begin{equation}
 \begin{array}{c}
-y_7y_8+y_0y_9+y_1y_{10}+y_4y_{11}, \ 
y_0y_5+y_1y_6-y_4y_7-y_{11}y_{12}, \ 
y_0y_2-y_4y_6-y_1y_7-y_{10}y_{12}, \\
-y_1y_2-y_4y_5-y_0y_7-y_9y_{12}, \ 
-y_0^2-y_1^2-y_4^2-y_8y_{12}, \  
-y_3y_8-y_9^2-y_{10}^2-y_{11}^2, \\
-y_3y_4-y_5y_9-y_6y_{10}-y_7y_{11}, \ 
-y_1y_3-y_2y_9-y_7y_{10}+y_6y_{11}, \ 
-y_0y_3-y_7y_9+y_2y_{10}+y_5y_{11}.
\end{array}
\end{equation}
Note that  $\sS\subset\PP^{12}$ 
is the intersection of a quadric hypersurface in $\PP^{12}$ 
with the cone over $\GG(1,4)\subset\PP^9\subset\PP^{12}$.
\end{example}
\begin{example}[$r=3;n=8;a=5$ - with non liftable inverse]\label{example: a=5}
If $X\subset\PP^8$ is the blow-up  of $\PP^3$  at a point $p$ with
 $|H_{X}|=|2H_{\PP^3}-p|$, then (modulo a change of coordinates)
the homogeneous ideal of $X$ is generated by the quadrics: 
\begin{equation}
 \begin{array}{c}
x_6x_7-x_5x_8,\ 
x_3x_7-x_2x_8,\ 
x_5x_6-x_4x_8,\ 
x_2x_6-x_1x_8,\ 
x_5^2-x_4x_7,\ 
x_3x_5-x_1x_8,\ 
x_2x_5-x_1x_7,\\
x_3x_4-x_1x_6,\ 
x_2x_4-x_1x_5,\ 
x_2x_3-x_0x_8,\ 
x_1x_3-x_0x_6,\ 
x_2^2-x_0x_7,\ 
x_1x_2-x_0x_5,\ 
x_1^2-x_0x_4 .
\end{array}
\end{equation}
The linear system $|\I_{X,\PP^8}(2)|$ defines a birational transformation 
$\psi:\PP^{8}\dashrightarrow\PP^{13}$ whose base locus is $X$ and whose 
image is the variety $\sS$ with homogeneous ideal generated by:
\begin{equation}
 \begin{array}{c}
y_8y_{10}-y_7y_{12}-y_3y_{13}+y_5y_{13},\ 
 y_8y_9+y_6y_{10}-y_7y_{11}-y_3y_{12}+y_1y_{13},\ 
 y_6y_9-y_5y_{11}+y_1y_{12},\\ 
 y_6y_7-y_5y_8-y_4y_{10}+y_2y_{12}-y_0y_{13},\  
 y_3y_6-y_5y_6+y_1y_8+y_4y_9-y_2y_{11}+y_0y_{12},\\ 
 y_3y_4-y_2y_6+y_0y_8,\ 
y_3^2y_5-y_3y_5^2+y_1y_3y_7-y_2y_3y_9+y_2y_5y_9-y_0y_7y_9-y_1y_2y_{10}+y_0y_5y_{10}. 
\end{array}
\end{equation}
We have $\deg(\sS)=19$, $\dim(\sing(\sS))=4$ and
the degrees of Segre classes of $X$ are: $s_1=-49$, $s_2=201$, $s_3=-627$. So,
by (\ref{eq: sollevabile}), we deduce that
 the inverse of 
 $\psi:\PP^8\dashrightarrow\sS$ is not liftable; 
however, a representative of the equivalence class of $\psi^{-1}$ is defined by:
\begin{equation}
 \begin{array}{c}
y_{12}^2-y_{11}y_{13},\  
y_8y_{12}-y_6y_{13},\  
y_8y_{11}-y_6y_{12},\  
-y_6y_{10}+y_7y_{11}+y_3y_{12}-y_5y_{12},\  
y_8^2-y_4y_{13},\\  
y_6y_8-y_4y_{12},\   
y_3y_8-y_2y_{12}+y_0y_{13},\  
y_6^2-y_4y_{11},\  
y_5y_6-y_1y_8-y_4y_9.
\end{array}
\end{equation}
We also point out that
  $\Sec(X)$ has dimension 
 $6$ and degree $6$ (against Proposition \ref{prop: B is QEL}).
\end{example}
\begin{example}[$r=3;n=8;a=6;d=2$]\label{example: 20}
 See also \cite{russo-simis} and \cite{semple}.
If $X=\PP_{\PP^1}(\O(1)\oplus\O(1)\oplus\O(4))$ or 
   $X=\PP_{\PP^1}(\O(1)\oplus\O(2)\oplus\O(3))$ or
   $X=\PP_{\PP^1}(\O(2)\oplus\O(2)\oplus\O(2))$, then 
$|\I_{X,\PP^8}(2)|$ defines a birational transformation
$\PP^8\dashrightarrow\sS\subset\PP^{14}$ of type $(2,2)$ 
whose base locus is $X$ and whose image is $\sS=\GG(1,5)$.
\end{example}
\begin{example}[$r=3; n=8; a=7; d=1$]\label{example: oadpDegree8}
See also \cite[Example~2.7]{ciliberto-mella-russo} and \cite{ionescu-smallinvariantsIII}.
Let $Z=\{p_1,\ldots,p_8\}\subset\PP^2$ be such that 
no $4$ of the $p_i$ are collinear and no $7$ of the $p_i$ lie on a conic and
consider the scroll $\PP_{\PP^2}(\mathcal{E})\subset\PP^7$ associated 
to the very ample vector bundle $\mathcal{E}$ of rank $2$,
given as an extension by the following exact sequence
$0\rightarrow\O_{\PP^2}\rightarrow\E\rightarrow 
\I_{Z,\PP^2}(4)\rightarrow0.$
The homogeneous ideal of $X\subset\PP^7$ is
generated by $7$ quadrics and so
the linear system $|\I_{X,\PP^8}(2)|$  defines a birational transformation 
$\psi:\PP^8\dashrightarrow\sS\subset\PP^{15}$ of type $(2,1)$.
Since we
have % $c_1(\mathcal{E})=4$, $c_2(\mathcal{E})=8$, 
$c_1(X)=12$,
$c_2(X)=15$,
$c_3(X)=6$, we deduce
$s_1(\mathcal{N}_{X,\PP^8})=-60$,
$s_2(\mathcal{N}_{X,\PP^8})=267$,
$s_3(\mathcal{N}_{X,\PP^8})=-909$,
and hence $\deg(\sS)=29$, by (\ref{eq: grado mappa razionale}).
The base locus  of the inverse of 
$\psi$ is $\psi(\PP^7)\simeq \PP^6\subset\sS\subset\PP^{15}$.
We also observe that the restriction of $\psi|_{\PP^7}:\PP^7\dashrightarrow\PP^6$
to a general hyperplane $H\simeq\PP^6\subset\PP^7$
gives rise to a transformation as in Example \ref{example: 6}.
\end{example}
\begin{example}[$r=3; n=8; a=8,9; d=1$]\label{example: edge}
If $X\subset\PP^7\subset\PP^8$ is a $3$-dimensional 
Edge variety of degree $7$ (resp. degree $6$), then 
$|\I_{X,\PP^8}(2)|$ defines a birational transformation
$\PP^8\dashrightarrow\sS\subset\PP^{16}$ 
(resp. $\PP^8\dashrightarrow\sS\subset\PP^{17}$) of type $(2,1)$ 
whose base locus is $X$ and 
whose degree of the image is $\deg(\sS)=33$ (resp. $\deg(\sS)=38$).
For memory overflow problems, 
we were not able to calculate the scheme $\sing(\sS)$;
however, it is easy to obtain that 
$1\leq \dim(\sing(\sS))<\dim(Y)=6$
and $\dim(\sing(Y))=1$, where $Y$ denotes the base locus of the inverse.
\end{example}
\begin{example}[$r=3; n=8; a=10; d=1$]\label{example: oadp10}
See also \cite{russo-simis}, \cite{semple} and \cite[\Rmnum{3} Theorem~3.8]{zak-tangent}.
We have a birational transformation
$\PP^{10}\dashrightarrow\GG(1,6)\subset\PP^{20}$ of type $(2,1)$
whose base locus 
is $\PP^1\times\PP^4\subset\PP^9\subset\PP^{10}$ and whose image is $\GG(1,6)$.
Restricting it
  to a general  $\PP^8\subset\PP^{10}$ 
 we obtain a birational transformation 
 $\psi:\PP^8\dashrightarrow\sS\subset\PP^{18}$
whose base locus $X$ is
a rational normal scroll 
(hence either $X=\PP_{\PP^1}(\O(1)\oplus\O(1)\oplus\O(3))$ 
or  $X=\PP_{\PP^1}(\O(1)\oplus\O(2)\oplus\O(2))$)
and whose image $\sS$ is  a linear section of
 $\GG(1,6)\subset\PP^{20}$. 
We denote by $Y\subset\sS$ the base locus of the inverse of $\psi$ and by  
$F=(F_0,\ldots,F_9):\PP^7\dashrightarrow\PP^9$ 
the restriction of $\psi$ to $\PP^7=\Sec(X)$. 
We have
\begin{eqnarray*}
Y&=&\overline{\psi(\PP^7)}=\overline{F(\PP^7)}=\GG(1,4)\subset\PP^9\subset\PP^{18} , \\
J_6&:=&\left\{x=[x_0,...,x_7]\in\PP^7\setminus X: \mathrm{rank}\left(\left({\partial F_i}/{\partial x_j}(x)\right)_{i,j}\right)\leq 6   \right\}_{\mathrm{red}}\\
    &=& \left\{x=[x_0,...,x_7]\in\PP^7\setminus X: \dim\left(\overline{F^{-1}\left(F(x)\right)}\right)\geq2   \right\}_{\mathrm{red}}\mbox{ and }\dim\left(J_6\right) = 5,\\
\overline{\psi\left(J_6\right)} &=& \left(\sing\left(\sS\right)\right)_{\mathrm{red}} \subset Y\mbox{ and }\dim\left(\overline{\psi\left(J_6\right)}\right) = 3. \\
\end{eqnarray*}
\end{example}
\section{Summary results}\label{sec: table}
\begin{theorem}\label{theorem: classification}
Table \ref{tabella: all cases 3-fold}
classifies all special quadratic transformations $\varphi$ 
as in \S \ref{sec: notation} and with $r\leq3$.
\end{theorem}
As a consequence, 
we generalize \cite[Corollary~6.8]{note}.
\begin{corollary}\label{corollary: coindex 2}
 Let $\varphi:\PP^n\dashrightarrow\sS\subseteq\PP^{n+a}$ be 
as in \S \ref{sec: notation}. If $\varphi$ is of type $(2,3)$ and 
$\sS$ has coindex $c=2$,
then $n=8$, $r=3$ and one of the following cases holds: 
\begin{itemize}
\item $\Delta=3$, $a=1$, $\lambda=11$, $g=5$, $\B$ is the blow-up of $Q^3$ at $5$ points; % $p_1,\ldots,p_5$  and  $|H_{\B}|=|2H_{Q^3}-p_1-\cdots-p_5|$; 
\item $\Delta=4$, $a=2$, $\lambda=10$, $g=4$, $\B$ is a scroll over $Q^2$;  
\item $\Delta=5$, $a=3$, $\lambda=9$,  $g=3$, $\B$ is a quadric fibration over $\PP^1$. 
\end{itemize}
\end{corollary}
\begin{proof}
 We have that $\B\subset\PP^n$ is a $QEL$-variety of type
$\delta=(r-d-c+2)/d=(r-3)/3$ and $n=((2d-1)r+3d+c-2)/d=(5r+9)/3$.
From Divisibility Theorem \cite[Theorem~2.8]{russo-qel1}, we deduce 
$(r,n,\delta)\in\{(3,8,0),(6,13,1),(9,18,2)\}$ and 
from the classification of $CC$-manifolds \cite[Theorem~2.2]{ionescu-russo-conicconnected}, 
we obtain  $(r,n,\delta)=(3,8,0)$. Now we apply the results in \S \ref{sec: dim 3}.
\end{proof}
We can also regard Corollary \ref{corollary: coindex 2} 
in the same spirit of \cite[Theorem~5.1]{note}, 
where  we have classified the transformations $\varphi$ 
of type $(2,2)$, when $\sS$ has coindex $1$.
Moreover, in the same fashion, one can prove the following:
\begin{proposition}
 Let $\varphi$ be as in \S \ref{sec: notation} and of type $(2,1)$. If
$c=2$, then $r\geq1$ and $\B$ is $\PP^1\times\PP^2\subset\PP^5$ 
or one of its linear sections.
If  $c=3$, then $r\geq2$ and $\B$ is either 
   $\PP^1\times\PP^3\subset\PP^7$ 
or $\GG(1,4)\subset\PP^9$
or one of their linear sections.
If $c=4$, then $r\geq3$ and $\B$ is either an $OADP$ $3$-fold in $\PP^7$ 
or $\PP^1\times\PP^4\subset\PP^{9}$ 
or one of its hyperplane sections.
\end{proposition}
In Table \ref{tabella: all cases 3-fold} we use the following shortcuts:
\begin{description}
\item[$\exists^{\ast}$] flags cases for which
is known a transformation $\varphi$ with base locus $\B$ as required,
but we do not know if the image $\sS$
satisfies all the assumptions in \S \ref{sec: notation};
\item[$\exists^{\ast\ast}$] flags cases for which
 is known that there is a smooth irreducible variety $X\subset\PP^n$ 
such that, if $X=V(H^0(\I_X(2)))$, then the linear system $|\I_X(2)|$ defines 
a birational  transformation 
$\varphi:\PP^n\dashrightarrow\sS=\overline{\varphi(\PP^n)}\subset\PP^{n+a}$ 
as stated;
\item[$?$] flags cases for which we do not know if there exists at least
an abstract variety
$\B$ having the structure and the invariants required;
\item[$\exists$] flags cases for which everything works fine.
\end{description}
\begin{table}
\begin{center}
\begin{tabular}{|c||c|c|c|c|c|c|c|c||ll|}
\hline
$r$ & $n$ & $a$ & $\lambda$ & $g$  & Abstract structure of $\B$ & $d$ & $\Delta$ & $c$ & \multicolumn{2}{|c|}{Existence}   \\  
\hline
\hline
\multirow{4}{*}{$1$} & $3$ & $1$ & $2$ & $0$ &  $\nu_2(\PP^1)\subset\PP^2$ & $1$ & $2$ & $1$ & $\exists$ & Ex. \ref{example: 1} \\
\cline{2-11}
 & $4$ & $0$ & $5$ & $1$ & Elliptic curve & $3$ & $1$ & $0$ & $\exists$ & Ex. \ref{example: 2} \\
\cline{2-11}
 & $4$ & $1$ & $4$ & $0$ & $\nu_4(\PP^1)\subset\PP^4$ & $2$ & $2$ & $1$ & $\exists$ & Ex. \ref{example: 3} \\ 
\cline{2-11}
 & $4$ & $3$ & $3$ & $0$ & $\nu_3(\PP^1)\subset\PP^3$ & $1$ & $5$ & $2$ & $\exists$ & Ex. \ref{example: 5} \\
\hline
\hline
\multirow{14}{*}{$2$} & $4$ & $1$ & $2$ & $0$  &  $\PP^1\times\PP^1\subset\PP^3$ & $1$ & $2$ & $1$ & $\exists$ & Ex. \ref{example: 1} \\ 
\cline{2-11}
 & $5$ & $0$ & $4$ & $0$  &  $\nu_2(\PP^2)\subset\PP^5$ & $2$ & $1$ & $0$ & $\exists$ & Ex. \ref{example: 3} \\ 
\cline{2-11}
 & $5$ & $3$ & $3$ & $0$  & Hyperplane section of $\PP^1\times\PP^2\subset\PP^5$ & $1$ & $5$ & $2$ & $\exists$ & Ex. \ref{example: 5} \\ 
\cline{2-11}
 & $6$ & $0$ & $7$ & $1$  & Elliptic scroll $\PP_{C}(\E)$ with $e(\E)=-1$ & $4$ & $1$ & $0$& $\exists$ & Ex. \ref{example: 6} \\ 
\cline{2-11}
 & $6$ & $0$ & $8$ & $3$  & \begin{tabular}{c} Blow-up of $\PP^2$ at $8$ points $p_1,\ldots,p_8$,\\ $|H_{\B}|=|4H_{\PP^2}-p_1-\cdots-p_8|$ \end{tabular} & $4$ & $1$ & $0$&$\exists$ & Ex. \ref{example: 6} \\ 
\cline{2-11}
 & $6$ & $1$ & $7$ & $2$  & \begin{tabular}{c} Blow-up of $\PP^2$ at $6$ points $p_0,\ldots,p_5$,\\ $|H_{\B}|=|4H_{\PP^2}-2p_0-p_1-\cdots-p_5|$ \end{tabular} & $3$ & $2$ & $1$& $\exists$ & Ex. \ref{example: 7} \\ 
\cline{2-11}
 & $6$ & $2$ & $6$ & $1$  & \begin{tabular}{c} Blow-up of $\PP^2$ at $3$ points $p_1,p_2,p_3$,\\ $|H_{\B}|=|3H_{\PP^2}-p_1-p_2-p_3|$ \end{tabular} & $2$ & $4$ & $2$& $\exists$ & Ex. \ref{example: 8} \\ 
\cline{2-11}
 & $6$ & $3$ & $5$ & $0$  & $\PP_{\PP^1}(\O(1)\oplus\O(4))$ or $\PP_{\PP^1}(\O(2)\oplus\O(3))$   & $2$ & $5$ & $2$ & $\exists$ & Ex. \ref{example: 9} \\ 
\cline{2-11}
 & $6$ & $5$ & $5$ & $1$  & \begin{tabular}{c} Blow-up of $\PP^2$ at $4$ points $p_1\ldots,p_4$,\\ $|H_{\B}|=|3H_{\PP^2}-p_1-\cdots-p_4|$ \end{tabular} & $1$ & $12$ & $3$ & $\exists$ & Ex. \ref{example: 10} \\ 
\cline{2-11}
 & $6$ & $6$ & $4$ & $0$  & $\PP_{\PP^1}(\O(1)\oplus\O(3))$ or $\PP_{\PP^1}(\O(2)\oplus\O(2))$ & $1$ & $14$ & $3$ & $\exists$ & Ex. \ref{example: 11} \\ 
\hline
\hline
\multirow{23}{*}{$3$} & $5$ & $1$ & $2$ & $0$  & $Q^3\subset\PP^4$ & $1$ & $2$ & $1$&$\exists$ & Ex. \ref{example: 1} \\
\cline{2-11}
 & $6$ & $3$ & $3$ & $0$  & $\PP^1\times\PP^2\subset\PP^5$ & $1$ & $5$ & $2$& $\exists$ & Ex. \ref{example: 5} \\
\cline{2-11}
 & $7$ & $1$ & $6$ & $1$  & Hyperplane section of $\PP^2\times\PP^2\subset\PP^8$ & $2$ & $2$ & $1$& $\exists$ & Ex. \ref{example: 3} \\
\cline{2-11}
 & $7$ & $5$ & $5$ & $1$  & Linear section of $\GG(1,4)\subset\PP^9$ & $1$ & $12$ &$3$ & $\exists$ & Ex. \ref{example: 10} \\ 
\cline{2-11}
 & $7$ & $6$ & $4$ & $0$  & $\PP_{\PP^1}(\O(1)\oplus\O(1)\oplus\O(2))$ & $1$ & $14$ &$3$ & $\exists$ & Ex. \ref{example: 11} \\
\cline{2-11}
 & $8$ & $0$ & $12$ & $6$    & \begin{tabular}{c} Scroll $\PP_{Y}(\E)$, $Y$ birat. ruled surface, \\ $K_Y^2=5$, $c_2(\E)=8$, $c_1^2(\E)=20$ \end{tabular} & $5$ & $1$ & $0$ & $?$ & \\
\cline{2-11}
 & $8$ & $0$ & $13$ & $8$    & \begin{tabular}{c} Variety obtained as the projection \\ of a Fano variety $X$ from a point $p\in X$ \end{tabular} & $5$ & $1$ & $0$ & $\exists$ & Ex. \ref{example: 12} \\
\cline{2-11}
 & $8$ & $1$ &  $11$   &  $5$  & \begin{tabular}{c} Blow-up of $Q^3$ at $5$ points $p_1,\ldots,p_5$, \\ $|H_{\B}|=|2H_{Q^3}-p_1-\cdots-p_5|$  \end{tabular} & $3$ &  $3$ & $2$ & $\exists$ &  Ex. \ref{example: 13} \\
\cline{2-11}
 & $8$ & $1$ &  $11$   &  $5$  & Scroll over $\PP_{\PP^1}(\O\oplus\O(-1))$  & $4$ &  $2$& $1$ & $\exists^{\ast\ast}$ & Ex. \ref{example: 14} \\
\cline{2-11}
 & $8$ & $1$ &  $12$   &  $7$ & Linear section of $S^{10}\subset\PP^{15}$  & $4$ & $2$& $1$ &$\exists$ & Ex. \ref{example: 15} \\
\cline{2-11}
 & $8$ & $2$ & $10$ & $4$ & Scroll over $Q^2$ & $3$ & $4$ & $2$ & $\exists^{\ast}$ & Ex. \ref{example: 16} \\
\cline{2-11}
 & $8$ & $3$ & $9$ & $3$ & Scroll over $\PP^2$ & $2$ & $8$ & $3$ & $\exists^{\ast}$ & Ex. \ref{example: 17} \\
\cline{2-11}
 & $8$ & $3$ & $9$ & $3$ & Quadric fibration over $\PP^1$ & $3$ & $5$ & $2$ & $\exists^{\ast}$ & Ex. \ref{example: 17} \\
\cline{2-11}
 & $8$ &$4$ & $8$ & $2$ & Hyperplane section of $\PP^1\times Q^3$ & $2$ & $10$ & $3$ & $\exists^{\ast}$ & Ex. \ref{example: 18}  \\
\cline{2-11}
 & $8$ &$6$ & $6$ & $0$ & Rational normal scroll & $2$ & $14$ & $3$ & $\exists$ &  Ex. \ref{example: 20} \\
\cline{2-11}
 & $8$ & $7$ & $8$ & $3$  & \begin{tabular}{c} $\PP_{\PP^2}(\E)$, where  $0\rightarrow\O_{\PP^2}\rightarrow$ \\$\rightarrow\E\rightarrow \I_{\{p_1,\ldots,p_8\},\PP^2}(4)\rightarrow0$ \end{tabular}  & $1$ & $29$ & $4$ & $\exists^{\ast}$ &  Ex. \ref{example: oadpDegree8} \\
\cline{2-11}
 & $8$ & $8$ & $7$ & $2$  & Edge variety & $1$ & $33$ & $4$ & $\exists^{\ast}$ & Ex. \ref{example: edge} \\
\cline{2-11}
 & $8$ & $9$ & $6$ & $1$  & $\PP^1\times\PP^1\times\PP^1\subset\PP^7$ & $1$ & $38$ & $4$ & $\exists^{\ast}$ & Ex. \ref{example: edge} \\
\cline{2-11}
 & $8$ & $10$ & $5$ & $0$  & Rational normal scroll & $1$ & $42$ & $4$ & $\exists$ & Ex. \ref{example: oadp10} \\ % \begin{tabular}{c} $\PP_{\PP^1}(\O(1)\oplus\O(1)\oplus\O(3))$ or \\ $\PP_{\PP^1}(\O(1)\oplus\O(2)\oplus\O(2))$ \end{tabular} 
\hline
\end{tabular}
\end{center}
 \caption{All transformations $\varphi$ as in \S \ref{sec: notation} and with $r\leq3$}
\label{tabella: all cases 3-fold}
\end{table}
\clearpage
\section{Towards the case of dimension 4}\label{sec: dim 4}
In this section we treat the case in which $r=4$.
However, when $\delta=0$, 
we are well away from having an exhaustive classification.

Proposition \ref{prop: delta mag0 4fold} follows from  
\cite[Propositions~1.3, 3.4, Corollary~3.2]{russo-qel1}  
and \cite[Theorem~2.2]{ionescu-russo-conicconnected}.
\begin{proposition}\label{prop: delta mag0 4fold} 
If $r=4$, then either $n=10$, $d\geq2$, $\langle \B \rangle = \PP^{10}$, or  
one of the following cases holds:
\begin{itemize}
 \item  $n=6$,  $d=1$, $\delta=4$, $\B=Q^4\subset\PP^5$ is a quadric;
 \item  $n=8$,  $d=1$, $\delta=2$, $\B\subset\PP^7$ is either $\PP^1\times\PP^3\subset\PP^7$ or a linear section of $\GG(1,4)\subset\PP^9$;
 \item  $n=8$,  $d=2$, $\delta=2$, $\B$ is $\PP^2\times\PP^2\subset\PP^8$;
 \item  $n=9$,  $d=1$, $\delta=1$, $\B$ is a hyperplane section of $\PP^1\times\PP^4\subset\PP^9$;
 \item  $n=10$, $d=1$, $\delta=0$, $\B\subset\PP^9$ is an $OADP$-variety. 
\end{itemize}
\end{proposition}
In Proposition \ref{prop: 4foldnondegenerate}, 
 we more generally assume that the image $\sS$ is  
nondegenerate, normal and linearly normal (not necessarily factorial) 
and furthermore we do not assume Assumptions 
\ref{assumption: liftable} and \ref{assumption: ipotesi}.
As noted earlier, we have $P_{\B}(1)=11$ and $P_{\B}(2)=55-a$ and hence 
\begin{eqnarray*}
P_{\B}(t)&=& \lambda \begin{pmatrix}t+3\cr 4\end{pmatrix}+\left( 1-g\right)  \begin{pmatrix}t+2\cr 3\end{pmatrix}+\left(2g-3\lambda+\chi(\O_{\B})-a+31 \right) \begin{pmatrix}t+1\cr 2\end{pmatrix} \\
 && +\left(-g+2\lambda-2\chi(\O_{\B})+a-21\right) t+\chi(\O_{\B}) .
\end{eqnarray*}
\begin{proposition}\label{prop: 4foldnondegenerate}
If $r=4$, $n=10$ and $\langle \B\rangle=\PP^{10}$, 
then  one of the following cases holds:
\begin{itemize}
 \item $a=10$, $\lambda=7$,             $g=0$,      $\chi(\O_{\B})=1$,                 $\B$ is a rational normal scroll;
 \item $a=7$,  $\lambda=10$,            $g=3$,      $\chi(\O_{\B})=1$,                 $\B$  is either
  \begin{itemize}
  \item a hyperplane section of $\PP^1\times Q^4\subset\PP^{11}$ or
  \item $\PP(\T_{\PP^2}\oplus \O_{\PP^2}(1))\subset\PP^{10}$;
  \end{itemize}
 \item $a=6$,  $\lambda=11$,            $g=4$,      $\chi(\O_{\B})=1$,                 $\B$ is a quadric fibration over $\PP^1$;
 \item $a=5$,  $\lambda=12$,            $g=5$,      $\chi(\O_{\B})=1$,                 $\B$ is one of the following:
\begin{itemize}
      \item $\PP^4$ blown up at $4$ points $p_1\ldots,p_4$ embedded by $|2H_{\PP^4}-p_1-\cdots-p_4|$,
      \item a scroll over a ruled surface,
      \item a quadric fibration over $\PP^1$; 
\end{itemize}         
 \item $a=4$,  $\lambda=14$,            $g=8$,      $\chi(\O_{\B})=1$,                 $\B$ is either
 \begin{itemize}
 \item a linear section of $\GG(1,5)\subset\PP^{14}$ or 
 \item the product of $\PP^1$ with a Fano variety of even index;  
 \end{itemize}
 \item $a=4$,  $\lambda=13$,            $g=6$,      $\chi(\O_{\B})=1$,                 $\B$ is either  
\begin{itemize}
      \item a scroll over a birationally ruled surface or
      \item a quadric fibration over $\PP^1$; 
\end{itemize} 
 \item $a=3$,  $14\leq\lambda\leq 16$,  $g\leq 11$, $\chi(\O_{\B})=(-g+2\lambda-18)/3$;
 \item $a=2$,  $15\leq \lambda\leq 18$, $g\leq 14$, $\chi(\O_{\B})=(-g+2\lambda-19)/3$; 
% \item $a\leq1$, $\lambda\geq15$.
 \item $a=1$, $15\leq\lambda\leq 20$, $g\leq17$, $\chi(\O_{\B})=(-g+2\lambda-20)/3$;
 \item $a=0$, $15\leq\lambda$. 
\end{itemize}
\end{proposition}
\begin{proof}
Denote by $\Lambda\subsetneq C\subsetneq S\subsetneq X\subsetneq \B$ a sequence of 
general linear sections of $\B$ and put  
$h_{\Lambda}(2):=h^0(\PP^6,\O(2))-h^0(\PP^6,\I_{\Lambda}(2))$.
Since $C$ is a nondegenerate curve in $\PP^7$, 
 we have $\lambda\geq 7$. 
By Castelnuovo's argument \cite[Lemma~6.1]{note}, 
it follows that
\begin{equation}\label{eq: castelnuovo-argument-4fold}
7\leq \min\{\lambda,13\}\leq h_{\Lambda}(2)\leq 28 - h^0(\PP^{10},\I_{\B}(2))=17-a 
\end{equation}
and in particular we have $a\leq 10$.
Moreover
\begin{itemize}
 \item if $\lambda\geq 13$, then $h_{\Lambda}(2)\geq 13$ and $a\leq4$, by (\ref{eq: castelnuovo-argument-4fold});
 \item if $\lambda\geq 15$, then $h_{\Lambda}(2)\geq 14$ and $a\leq3$, by Castelnuovo Lemma \cite[Lemma~1.10]{ciliberto-hilbertfunctions};
 \item if $\lambda\geq 17$, then $h_{\Lambda}(2)\geq 15$ and $a\leq2$, by \cite[Theorem~3.1]{ciliberto-hilbertfunctions};
 \item if $\lambda\geq 19$, then $h_{\Lambda}(2)\geq 16$ and $a\leq1$, by \cite[Theorem~3.8]{ciliberto-hilbertfunctions};
 \item if $\lambda\geq 21$, then $h_{\Lambda}(2)\geq 17$ and $a=0$, by \cite[Theorem~2.17(b)]{petrakiev}. 
\end{itemize}
According to the above statements,
we consider the refinement 
$\theta=\theta(\lambda)$ of Castelnuovo's bound $\rho=\rho(\lambda)$, contained in 
 \cite[Theorem~2.5]{ciliberto-hilbertfunctions}.
So, we have
\begin{equation}\label{eq: KBHB3}
K_{\B}\cdot H_{\B}^3=2g-2-3\lambda \leq 2\theta(\lambda)-2-3\lambda\leq 2\rho(\lambda)-2-3\lambda .
\end{equation}
Now, if $t\geq1$, by Kodaira Vanishing Theorem and Serre Duality,
it follows that $P_{\B}(-t)=h^4(\B,\O_{\B}(-t))=h^0(\B,K_{\B}+tH_{\B}) $;
hence, if $P_{\B}(-t)\neq0$, then $K_{\B}+tH_{\B}$ is an effective divisor and 
 we have either $K_{\B}\cdot H_{\B}^3 > -tH_{\B}^4=-t \lambda$ or $K_{\B}\sim -tH_{\B}$.
Thus, by (\ref{eq: KBHB3}) and  
 straightforward calculation, we deduce (see Figure \ref{fig: upperbounds}):
\begin{figure}[htb] 
\centering
\includegraphics[width=0.7\textwidth]{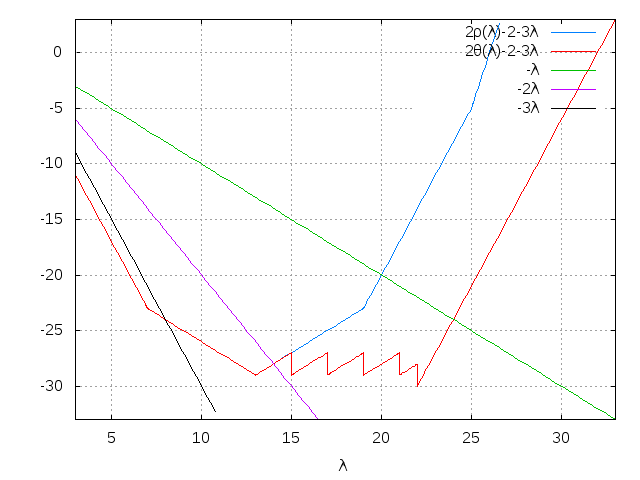} 
\caption{Upper bounds of $K_{\B}\cdot H_{\B}^3$}
\label{fig: upperbounds}
\end{figure}
\begin{enumerate}[(\ref{prop: 4foldnondegenerate}.a)]
 \item\label{case: lambda8}  if $\lambda\leq8$,  then either $P_{\B}(-3)=P_{\B}(-2)=P_{\B}(-1)=0$ or $\lambda=8$ and $K_{\B}\sim -3H_{\B}$;
 \item\label{case: lambda14} if $\lambda\leq14$, then either $P_{\B}(-2)=P_{\B}(-1)=0$ or $\lambda=14$ and $K_{\B}\sim -2H_{\B}$; 
% \item\label{case: lambda20} if $\lambda\leq23$, then either $P_{\B}(-1)=0$ or $\lambda=23$ and $K_{\B}\sim -H_{\B}$.
 \item\label{case: lambda20} if $\lambda\leq24$, then either $P_{\B}(-1)=0$ or $\lambda=24$ and $K_{\B}\sim -H_{\B}$.
\end{enumerate}
In the same way, one also sees that $h^4(\B,\O_{\B})=0$ 
% whenever $\lambda\leq 30$.
whenever $\lambda\leq 31$.
Now we discuss the cases according to the value of $a$.
\begin{case}[$9\leq a\leq 10$] We have $\lambda\leq 8$.
From the classification of del Pezzo varieties 
in \cite[\Rmnum{1} \S 8]{fujita-polarizedvarieties},
we see that the case $\lambda=8$ with $K_{\B}\sim -3H_{\B}$  is impossible and so
we obtain $\lambda=11-2a/5$, $g=1-a/10$, 
by (\ref{prop: 4foldnondegenerate}.\ref{case: lambda8}).
Hence $a=10$, $\lambda=7$, $g=0$ and $\B$ is a rational normal scroll.
\end{case}
\begin{case}[$5\leq a\leq 8$] We have $\lambda\leq 12$.
  By (\ref{prop: 4foldnondegenerate}.\ref{case: lambda14}) we obtain 
 $g=(3\lambda+a-31)/2$ and $\chi(\O_{\B})=(\lambda+a-11)/6$ and,
since  $\chi(\O_{\B})\in \ZZ$,
we obtain $\lambda=17-a$, $g=10-a$, $\chi(\O_{\B})=1$. So,
we can determine the abstract structure of $\B$ by
\cite{fania-livorni-ten},
\cite{besana-biancofiore-deg11},
\cite[Theorem~2]{ionescu-smallinvariantsII},
\cite[Lemmas~4.1 and 6.1]{besana-biancofiore-numerical}
 and  we also deduce that the case $a=8$ does not occur, by \cite{fania-livorni-nine}.
\end{case}
\begin{case}[$a = 4$] We have $\lambda\leq 14$. 
 Again by (\ref{prop: 4foldnondegenerate}.\ref{case: lambda14}), we deduce that either
 $g=(3\lambda-27)/2$ and $\chi(\O_{\B})=(\lambda-7)/6$ or
 $\B$ is a Mukai variety with 
$\lambda=14$ ($g=8$ and $\chi(\O_{\B})=1$).
In the first case, since
$\chi(\O_{\B})\in \ZZ$ and $g\geq 0$,
we obtain $\lambda=13$, $g=6$, $\chi(\O_{\B})=1$ and 
then we can determine the abstract structure of $\B$ by 
\cite[Theorem~1]{ionescu-degsmallrespectcodim} and
\cite[Lemmas~4.1 and 6.1]{besana-biancofiore-numerical}.
In the second case, if $b_2=b_2(\B)=1$ then $\B$ is a linear section 
of $\GG(1,5)\subset\PP^{14}$, 
otherwise $\B$ is a Fano variety of product type, see
 \cite[Theorems~2 and 7]{mukai-biregularclassification}.
\end{case}
\begin{case}[$a=3$] We have $\lambda\leq 16$ and 
$\chi(\O_{\B})=(-g+2\lambda-18)/3$, 
by (\ref{prop: 4foldnondegenerate}.\ref{case: lambda20}).
Moreover, if $\lambda\leq14$, 
by (\ref{prop: 4foldnondegenerate}.\ref{case: lambda14}) it follows
that $\lambda=14$, $g=7$ and $\chi(\O_{\B})=1$.
\end{case}
\begin{case}[$a=2$] We have $\lambda\leq 18$   
and $\chi(\O_{\B})=(-g+2\lambda-19)/3$, 
by (\ref{prop: 4foldnondegenerate}.\ref{case: lambda20}).
Moreover, by (\ref{prop: 4foldnondegenerate}.\ref{case: lambda14}) it follows that
$\lambda\geq 15$.
\end{case}
%  \begin{case}[$a\leq1$] If $\lambda\leq 14$,   
%  by (\ref{prop: 4foldnondegenerate}.\ref{case: lambda14})
%  and (\ref{prop: 4foldnondegenerate}.\ref{case: lambda20}) 
%  it follows that either
%  $a=1$, $\lambda=10$, $g=0$, $\chi(\O_{\B})=0$ or
%  $a=0$, $\lambda=11$, $g=1$, $\chi(\O_{\B})=0$. 
%  The first case is, of course, impossible.
%  In the second case, $\B$ must be an elliptic scroll  
%  and $\varphi$ must be of type $(2,6)$; so, 
%  by (\ref{eq: c2 4-fold}) we obtain the contradiction
%  $c_2(\B)\cdot H_{\B}^2=(990+c_4(\B))/37=990/37\notin\ZZ$.
%  \end{case}
\begin{case}[$a=1$] 
We have $\lambda\leq 20$ and 
$\chi(\O_{\B})=(-g+2\lambda-20)/3$, 
by (\ref{prop: 4foldnondegenerate}.\ref{case: lambda20}).
Moreover, if $\lambda\leq14$, 
by (\ref{prop: 4foldnondegenerate}.\ref{case: lambda14}) it follows
that  $\lambda=10$, $g=0$, $\chi(\O_{\B})=0$,
which is of course impossible.
\end{case}
\begin{case}[$a=0$]
If $\lambda\leq 14$,   
by (\ref{prop: 4foldnondegenerate}.\ref{case: lambda14})
and (\ref{prop: 4foldnondegenerate}.\ref{case: lambda20}) 
it follows that  $\lambda=11$, $g=1$, $\chi(\O_{\B})=0$. 
Thus, $\B$ must be an elliptic scroll  
and $\varphi$ must be of type $(2,6)$; so, 
by (\ref{eq: c2 4-fold}) we obtain the contradiction
$c_2(\B)\cdot H_{\B}^2=(990+c_4(\B))/37=990/37\notin\ZZ$.
\end{case}
\end{proof}
\begin{remark}
Under the hypothesis of Proposition \ref{prop: 4foldnondegenerate},
 reasoning as in Proposition \ref{prop: segre and chern classes}, 
we obtain that
if $\varphi$ is of type $(2,d)$, then
 \begin{eqnarray}
\label{eq: c2 4-fold}  37c_2(\B)\cdot H_{\B}^2-c_4(\B) &=&  -231\lambda+188g+(1-9d)\Delta+3396 ,\\
  37c_3(\B)\cdot H_{\B}+7c_4(\B) &=& 655\lambda-428g+(26d-7)\Delta-5716 .
 \end{eqnarray}
\end{remark}
\begin{remark}
If Eisenbud-Green-Harris Conjecture $I_{11,6}$ holds 
(see \cite{eisenbud-green-harris}),
 then we have that $\lambda\leq 24$,
even in the case with $a=0$.
% of Proposition \ref{prop: 4foldnondegenerate}.
If $a=0$ and $\lambda\leq 24$, we have $g\leq \theta(24)=25$ and 
one of the following cases holds:
\begin{itemize}
 \item $\lambda=24$, $g=25$, $\chi(\O_{\B})=1$ and $\B$ is a Fano variety of coindex $4$; % P_B(t)=(24*t^4+48*t^3+96*t^2+72*t+24)/24 
 \item $g\leq 24$ and $\chi(\O_{\B})=(-g+2\lambda-21)/3$.
\end{itemize}
\end{remark}
\begin{example} 
Note that 
in Proposition \ref{prop: delta mag0 4fold}, all cases  with $\delta>0$ 
really occur (see \S \ref{sec: examples});
when $\delta=0$,  
an example is obtained by taking a
general $4$-dimensional linear section of 
$\PP^1\times\PP^5\subset\PP^{11}\subset\PP^{12}$.
Below we collect some examples of 
special quadratic birational transformations appearing in Proposition
\ref{prop: 4foldnondegenerate}.
\begin{itemize}
\item If $X\subset\PP^{10}$ is a (smooth) $4$-dimensional rational normal scroll, 
then $|\I_{X,\PP^{10}}(2)|$
defines a birational transformation 
$\psi:\PP^{10}\dashrightarrow\GG(1,6)\subset\PP^{20}$ of type $(2,2)$.
\item If $X\subset\PP^{10}$ is a general hyperplane section 
of $\PP^1\times Q^4\subset\PP^{11}$,
then $|\I_{X,\PP^{10}}(2)|$ defines a birational transformation 
$\psi:\PP^{10}\dashrightarrow \overline{\psi(\PP^{10})}\subset\PP^{17}$
of type $(2,2)$ whose image has degree $28$.
\item If $X=\PP(\T_{\PP^2}\oplus\O_{\PP^2}(1))\subset\PP^{10}$, since 
$h^1(X,\O_X)=h^1(\PP^2,\O_{\PP^2})=0$, $|\I_{X,\PP^{10}}(2)|$ defines a birational 
transformation 
$\psi:\PP^{10}\dashrightarrow \overline{\psi(\PP^{10})}\subset\PP^{17}$ 
(see Facts \ref{fact: test K2} and \ref{fact: K2 property}).
\item There exists a smooth linearly normal $4$-dimensional
variety $X\subset\PP^{10}$ with $h^1(X,\O_X)=0$, degree $11$, sectional genus $4$,
having the structure of a quadric fibration over $\PP^1$ 
(see \cite[Remark~3.2.5]{besana-biancofiore-deg11});
thus $|\I_{X,\PP^{10}}(2)|$ defines a birational 
transformation 
$\psi:\PP^{10}\dashrightarrow \overline{\psi(\PP^{10})}\subset\PP^{16}$
(see Facts \ref{fact: test K2} and \ref{fact: K2 property}).
\item If $X\subset\PP^{10}$ is the blow-up of $\PP^4$ at $4$ 
general points $p_1,\ldots,p_4$,
embedded by $|2H_{\PP^4}-p_1-\cdots-p_4|$, then
$|\I_{X,\PP^{10}}(2)|$ defines a birational transformation 
$\psi:\PP^{10}\dashrightarrow \overline{\psi(\PP^{10})}\subset\PP^{15}$
whose image has degree $29$; in this case $\Sec(X)$ is a
complete intersection of two cubics.
\item If $X\subset\PP^{10}$ is a general $4$-dimensional linear section of 
$\GG(1,5)\subset\PP^{14}$, then $|\I_{X,\PP^{10}}(2)|$ defines 
a birational transformation 
$\psi:\PP^{10}\dashrightarrow \overline{\psi(\PP^{10})}\subset\PP^{14}$
of type $(2,2)$ whose image is a complete intersection of quadrics.
\end{itemize}
\end{example}
\subsection*{Acknowledgements}
The author wishes to thank 
Prof. Francesco Russo 
for helpful discussions on this paper.

\bibliographystyle{abbrv}
\bibliography{bibliography.bib}

\end{document}